\documentclass[sn-mathphys-num]{sn-jnl}

\usepackage{graphicx}%
\usepackage{multirow}%
\usepackage{amsmath,amssymb,amsfonts}%
\usepackage{amsthm}%
\usepackage{mathrsfs}%
\usepackage[title]{appendix}%
\usepackage{xcolor}%
\usepackage{textcomp}%
\usepackage{manyfoot}%
\usepackage{booktabs}%
\usepackage{algorithm}%
\usepackage{algorithmicx}%
\usepackage{algpseudocode}%
\usepackage{listings}%

\theoremstyle{thmstyleone}%
\newtheorem{theorem}{Theorem}
\newtheorem{proposition}[theorem]{Proposition}%
\newtheorem{corollary}[theorem]{Corollary}
\newtheorem{lemma}[theorem]{Lemma}

\theoremstyle{thmstyletwo}%
\newtheorem{remark}{Remark}%

\theoremstyle{thmstylethree}%
\newtheorem{definition}{Definition}%

\newcommand{\p}{\partial}

\newcommand{\dd}{{\rm d}}

\begin{document}

\title[Global hyperbolicity and manifold topology]{Global hyperbolicity and manifold topology from the Lorentzian distance}


\author[1]{\fnm{Aleksei} \sur{Bykov}} \email{aleksei.bykov@alumni.uniroma2.eu}


\author[2]{\fnm{Ettore} \sur{Minguzzi} }\email{ettore.minguzzi@unipi.it}


\affil[1]{\orgdiv{Dipartimento di Matematica e Informatica ``U. Dini'', Universit\`a degli Studi di Firenze},  \orgaddress{\street{via S. Marta 3},  \postcode{50139}, \city{Firenze}, \country{Italy}
}}

\affil[2]{\orgdiv{Dipartimento di Matematica}, \orgname{Universit\`a degli Studi di Pisa}, \orgaddress{\street{Largo B. Pontecorvo 5}, \postcode{56127}, \city{Pisa}, \country{Italy}

}}


\abstract{In this work, we seek characterizations of global hyperbolicity in smooth Lorentzian manifolds that do not rely on the manifold topology and that are  inspired by metric geometry. In particular, strong causality is not assumed, so part of the problem is precisely that of  recovering the manifold topology so as to make sense of it also in rough frameworks. After verifying that known standard characterizations do not meet this requirement, we propose two possible formulations.
The first is based solely on chronological diamonds and is interesting due to its analogies with the Hopf-Rinow theorem. The second uses only properties of the Lorentzian distance function and it is suitable for extension to abstract `Lorentzian metric' frameworks. It turns out to be equivalent to the definition of `Lorentzian metric space' proposed in our previous joint work with S. Suhr, up to slightly strengthening  weak $d$-distinction to `future or past $d$-distinction'. The role of a new property which we term  {\em $d$-reflectivity} is also discussed. We then investigate continuity properties of the Lorentzian distance  and the property of $d$-reflectivity in non-smooth frameworks.
Finally, we establish a result of broader interest: the exponential map of a smooth spray is \( C^{1,1} \) (smooth outside the zero section). Additionally, we derive a Lorentz-Finsler version of the Busemann-Mayer formula and demonstrate that, in strongly causal smooth Finsler spacetimes, the Finsler fundamental function can be reconstructed from the distance. As a consequence, distance-preserving bijections are shown to be Lorentz-Finsler isometries in the conventional smooth sense.}


\maketitle

\section{Introduction}
The condition of global hyperbolicity is the most important in mathematical relativity, as it allows us to regard the evolution of spacetime as a well-defined PDE problem, with initial conditions imposed on a Cauchy hypersurface.

Many equivalent characterizations of global hyperbolicity have been proven over the years for smooth Lorentzian manifolds (all of which can be extended to $C^{1,1}$ metrics; we shall not concern ourselves here with the best regularity class). The many different, often disparate, ways of formulating the condition point to its ubiquity and fundamental physical importance.

The objective of this work is to formulate global hyperbolicity without appealing to the manifold topology. It is expected, particularly for unification with the quantum world, that at the most fundamental level spacetime is not described by a smooth manifold, as quantum mechanics usually introduces elements of discreteness. At the same time, causality and the related notion of global hyperbolicity appear too fundamental to be discarded lightly.

While we cannot be sure which mathematical structure will prove to be the correct one to replace the successful smooth Lorentzian manifold of general relativity, it seems reasonable to look for a formulation of global hyperbolicity that does not use the manifold topology. This is because the notion of order implicit in the causal structure, or the measure implicit in the spacetime volume, can be more easily substituted by other abstract mathematical entities than the manifold and its topology \cite{minguzzi13e}. In a sense, the manifold topology should be recovered and hence justified among the many available possibilities. Early approaches have, for instance,  explored its connection with alternative topologies derived from notable families of curves or hypersurfaces on spacetime \cite{hawking76,goebel76}, see
e.g.\ the discussion in \cite{low16,papadopoulos19,papadopoulos21}.

In this endeavor, we shall take inspiration from metric geometry \cite{burago01}, thus following the recent trend in Lorentzian metric/length spaces \cite{kunzinger18,muller19,ake20,cavalletti20,costa20,braun22,muller22,mccann22,braun23b,burgos23,mccann24,sakovich22,kunzinger22,minguzzi22,beran23}, rather than order/domain theory. In the latter direction, an important reference is the work by Martin and Panangaden \cite{martin06} and subsequent developments \cite{ebrahimi14,ebrahimi15,sharifzadeh19,finster21,mazibuko23}.
For instance, it has been shown \cite{ebrahimi15,sorkin19} that under $K$-causality (which is equivalent to stable causality) the manifold topology can be recovered from the $K$ (or Seifert) relation in a completely order theoretic manner using ideas from domain theory.
 Note that in these studies topology is derived from the notion of order, while in metric inspired theories topology is deduced from  a two-point function.

Spacetimes $(M,g)$, $g\in C^{1,1}$, admit a natural two-point function, the {\em Lorentzian distance} $d: M\times M\to [0,\infty]$. This is defined for $x,y\in M$ by the expression \cite{beem96,minguzzi18b}
\[
d(x,y):= \sup_\sigma \ell[\sigma], \qquad  \ell[\sigma]:=\int_\sigma \sqrt{-g(\dot \sigma, \dot \sigma}) \dd t ,
\]
where the supremum is over all  future-directed absolutely continuous causal curves $\sigma: I \to M$ connecting $x$ to $y$, and by setting $d(x,y)=0$ if there is no such curve (i.e.\ $(x,y)\notin J$ where $J$ is the causal relation).
Let us denote $d_p:=d(p, \cdot)$ and $d^p:=d(\cdot, p)$.
Ultimately, we prove that globally hyperbolic spacetimes and their topologies are characterized  as follows  (this is the same as Thms.\ \ref{cjpr}, \ref{coox} but with all properties spelled out).\\

\begin{theorem} \label{cakop}
A  spacetime $(M,g)$, $g\in C^{1,1}$, is globally hyperbolic if and only if the Lorentzian distance  has the following properties
\begin{itemize}
\item[(i)] $d$ is finite,
\item[(ii)]  there is a topology $T$ on $M$ such that $d$ is $T\times T$-continuous and the chronological diamonds  $I(p,q):=\{r: d_p(r) d^q(r)>0\}$,  $p,q\in M$, are relatively compact in the $T$ topology.
\item[(iii)] $(M,d)$ is future or past $d$-distinguishing, namely
\[
``\forall p,q\in M,  \ d_p=d_q \Rightarrow p=q " \textrm{ or } \ `` \forall p,q\in M,  \ d^p=d^q \Rightarrow p=q ".
\]
\end{itemize}
In this case $T$ is the manifold topology and the causal relation is given by
\begin{equation} \label{ckkqp}
J=\{(p,q): d_p \ge d_q\}=\{(p,q):  d^p\le d^q\}.
\end{equation}
\end{theorem}

Taking into account that $d$ satisfies the reverse triangle inequality for chronologically related events and that every event is included in some chronological diamond, we recognize that these properties are precisely those that define a Lorentzian metric space (without chronological boundary, $M=I(M)$) in the sense of our previous joint work with Suhr \cite{minguzzi22} \cite[Thm.\ 2.14]{minguzzi24b}, up to the fact that in a Lorentzian metric space (iii) is replaced by the slightly weaker `weak $d$-distinction': $\forall p,q\in M$, $d_p=d_q$ and $d^p=d^q \Rightarrow p=q$ (unfortunately, we do not know if (iii) can be weakened to such property).

The fact that properties (i)-(iii) are necessary was proved in \cite[Prop.\ 2.4, Thm.\ 2.14]{minguzzi24b} and is relatively straightforward, while the above type of converse is considerably more difficult and is only proved in this work.
The difficulties are related to the absence of any a priori relative strength between Lorentzian metric space topology (i.e.\ the topology $T$, which is actually known to be unique and coincident with the initial topology of the functions $d^p$, $d_p$, $p\in M$) and the manifold topology: neither is a priori included in the other.

For instance, we expect that discontinuities of $d$  in the manifold topology increase the family of $T$-open sets with respect to the manifold topology, while causal pathologies such as violation of strong causality decrease their number. Extreme care must be exerted in proofs, as continuous curves in one topology might not be continuous in the other, the time functions in the regular spacetime sense might not coincide with the time functions in the Lorentzian metric space sense, causal curves in the regular spacetime sense might  not coincide with isocausal curves in the Lorentzian metric space sense, and so on. Still, the fact that for both topologies a limit curve theorem holds, though relative to different type of curves and topologies, signals that they could be related and actually coincide. In fact, the validity of a limit curve type theorem has been recognized as one important criteria to select the natural spacetime topology \cite{low16,papadopoulos21}.

The above properties (i)-(iii) are formulated just through $d$, without including conditions such as $p\in \textrm{Cl}_T({I^\pm(p)})$, or the requirement that through each point passes an isochronal curve. This makes them suitable for an abstract definition of Lorentzian metric space which includes causal sets \cite{surya19}, as the latter are naturally endowed with the discrete topology. If such pointwise conditions were included,  the topology $T$ would coincide with the Alexandrov topology \cite[Thm.\ 3.9]{minguzzi24b} and the proof would be considerably simplified,\footnote{Some authors \cite{mondino18,mccann24} work with a function $\tau: M\times M\to \{-\infty\} \cup [0,+\infty]$ which coincides with $d$ only where it is not equal to $\{-\infty\}$, the  locus $\{d=-\infty\}$ being the complement of a closed relation $J$ identified with the causal relation. For a comparison of these approaches see \cite{minguzzi24b}. Essentially, in the latter framework $\tau$ contains more information than $d$ so it is easier to obtain from conditions on it various results. The additional information in $\tau$ is not strictly needed as the causal relation and its closure can be recovered with formulas such as (\ref{ckkqp}), see also \cite[Eq.\ (10)]{minguzzi22} \cite[Eq.\ (3)]{minguzzi24b}. These formulas are related to relation $D$ in the regular settings \cite{minguzzi07e,heveling21}, as we shall recall.} cf.\ Thm.\ \ref{cjpr}.

This work is structured as follows. We end this Introduction by recalling some characterizations of global hyperbolicity in subsection \ref{cnqkd}. In subsection \ref{mopr} we single out a formulation independent of topology, and dependent on properties of the family of chronological diamonds, which is interesting for analogies with the Heine-Borel property in positive signature. In Section \ref{cnpda} we approach the problem of characterizing global hyperbolicity via the Lorentzian distance function.  We introduce the interesting new notions of $d$-distinction and $d$-reflectivity, which turn out, in the regular case considered there, to be equivalent to the standard notions, having however the advantage of being formulated directly via the Lorentzian distance.
Ultimately, we prove the above main theorem, and in the following sections we demonstrate how the characterisations we suggest
work in less regular or Finslerian cases.

In Section \ref{cnqp4} we prove that globally hyperbolic proper Lorentz-Finsler spaces in the low regularity  sense of \cite{minguzzi17} are Lorentzian metric spaces in the sense of \cite{minguzzi22,minguzzi24b}, extending a result for regular spacetimes already contained in \cite{minguzzi24b}. In Section \ref{cmp3q} the folklore result that distance preserving maps of smooth spacetimes are actually Lorentzian isometries is proved. This is done to answer some questions in relation to the previous paper on Lorentzian metric spaces \cite{sakovich24,minguzzi22}. Subsequently, the result is generalized to the Lorentz-Finsler case by passing through the proof of  a Lorentz-Finsler Busemann-Mayer type formula (and another useful formula).
In Section \ref{qqpo} we pass to the low regularity theory first showing that previous results on the Lorentzian distance of Section \ref{cnpda} generalize to proper Lorentz-Finsler spaces. Then, several results on the Lorentzian distance are explored for structure even more general than Lorentzian metric spaces. In particular, the interplay between the continuity properties of $d$, $d$-reflectivity and the condition $p\in \overline{I^{\pm}(p)}$ is clarified.

Before we continue  let us recall our notations and terminologies.
In this work a manifold is always Hausdorff and second countable, hence paracompact. A spacetime $(M,g)$ is a connected time-oriented Lorentzian manifold whose metric $g$ is $C^2$ ($C^{1,1}$ will be enough) and of signature $(-,+,\cdots, +)$. It can also be called {\em regular spacetime}.
The inclusion $\subset$ is reflexive. When comparing topology `finer' shall always mean `as fine as', that is, `containing'. With a curve $\gamma$ we might mean a map  $\gamma\colon I \to M$ or the image of the map.
We write $p<q$ if there is a causal curve connecting $p$ to $q$, and $p\ll q$ if there is a timelike curve connecting $p$ to $q$. We write $p\le q$ if $p<q$ or $p=q$. The sets $J=\{(p,q)\colon p\le q\}$ and $I=\{(p,q)\colon p\ll q\}$ are the causal and chronological relations respectively (for a Lorentzian metric space the abstract causal relation might be denoted $\tilde J$ to avoid confusion with that for $(M,g)$).
A {\em causal diamond} is a set of the form $J^+(p)\cap J^-(q)$ also denoted $J(p,q)$ (and similarly for {\em chronological diamonds}).
The {\em causally convex hull} of a set $S\subset M$, is the set $J^+(S)\cap J^-(S)$, namely the union of the images of all the causal curves which start and end in $S$.
For most results of causality theory we refer to the recent review \cite{minguzzi18b}.

\subsection{Available characterizations} \label{cnqkd}

Let us consider some classical characterization of global hyperbolicity \cite{hawking73}
\begin{itemize}
\item[(a)] Strong causality and the causal diamonds are compact,
\item[(b)] Existence of a Cauchy hypersurface \cite{geroch70},
\item[(c)] Strong causality and compactness of the space of causal curves connecting any two events with respect to the $C^0$ topology \cite[Prop.\ 6.6.2]{hawking73}.
\end{itemize}

There are many variants of (a) in which strong causality is replaced by weaker conditions or in which the causal diamonds are replaced by different causal hulls  \cite{bernal06b,minguzzi19c,minguzzi23} \cite[Cor.\ 3.2, Thm.\ 3.6]{minguzzi08e} \cite{samann16} \cite[Def.\ 2.20]{minguzzi17} but all make use of compactness with respect to the manifold topology.

As for (b), we recall that a Cauchy hypersurface is an acausal hypersurface (without edge) that intersects every inextendible causal curve exactly once. Here the manifold topology enters both   the definition of hypersurface and curve. In particular, it seems difficult to dispense with it in the definition of curves.

Finally, in (c) the use of topology is evident as the topology on the space of curves is induced from that on the manifold.

To the previous classical list we can add
\begin{itemize}
\item[(d)] existence of a smooth $h$-steep Cauchy temporal function \cite[Thm.\ 3]{bernard18} (see also \cite[Thm.\ 2.45]{minguzzi17});
\item[(e)]  There is a time function $\tau$ such that $(M, \hat d_\tau)$ is  a complete metric space, where $\hat d_\tau$ is the null distance \cite[Thm.\ 1.4]{burtscher22}.
\end{itemize}

Here use of the manifold topology is made in the time/temporal functions used which, by definition, must be continuous in the manifold topology.

\subsection{A manifold-topology-independent characterization} \label{mopr}

Having established that none of the available characterizations of global hyperbolicity is independent of the manifold topology one might suspect that it should be difficult to find one. Instead,  there is a rather simple characterization that accomplishes such objective:\\

\begin{theorem} \label{prtr}
A spacetime $(M,g)$ is globally hyperbolic iff  the topology generated by the chronological diamonds is Hausdorff and any chronological diamond is relatively compact in that topology.\\
\end{theorem}

We recall that a chronological diamond is a set of the form $I(p,q):=I^+(p)\cap I^-(q)$ where $I$ is the chronological relation (which is known to be open in the manifold topology).
Observe that in a regular spacetime if a point is contained in the intersection of two diamonds then we can find a third diamond containing the point and staying within the two diamonds (just consider the endpoints of a short timelike curve passing through the point). This means that the family of chronological diamonds, regarded as a subbasis for a topology, is actually a basis for such topology. The topology so generated is called {\em Alexandrov topology}.

\begin{proof}%
The direction to the right is clear due to strong causality \cite{hawking66b,kronheimer67,penrose72}\cite[Thm.\ 4.75]{minguzzi18b} and the fact that chronological diamonds are contained within causal diamonds (which are compact by the standard definition of global hyperbolicity).

For the converse, the assumption that the Alexandrov topology is Hausdorff implies, by a classical result \cite{hawking66b,kronheimer67,penrose72}\cite[Thm.\ 4.75]{minguzzi18b}, that the spacetime is strongly causal and that the Alexandrov topology coincides with the manifold topology. In particular, the spacetime is non-total imprisoning, and thus the result follows from \cite[Cor.\ 3.3]{minguzzi08e}: a spacetime is globally hyperbolic if it is non-total imprisoning and, for every $p, q \in M$, $\overline{I(p,q)}$ is compact. In the original formulation, the closure was taken with respect to the manifold topology, which we have just shown to coincide with the Alexandrov topology. This concludes the proof.
\end{proof}

This characterization is remarkable in several respects. Not only does it achieve our objective, but it also depends solely on the notion of a chronological diamond. This notion contains the same information as the chronological relation itself; while we defined the chronological diamond using the chronological relation, the reverse direction can also be established via
\[
(p,q) \in I \iff I(p,q) \ne \emptyset.
\]
Of course,  one may insist that in the traditional regular setting of mathematical relativity the chronological relation must be derived from a notion of timelike curve, and that such curves must be continuous with respect to some topology, then it becomes difficult to accept this as a definition of global hyperbolicity that is independent of topology. But that is precisely the point: the result above suggests that timelike curves and topology are less primitive than the chronological relation. From this perspective, it makes sense to regard spacetime as constructed from a relation expressing causality, supplemented by other objects such as, e.g., a spacetime measure or a Lorentzian distance. The path from the relation to global hyperbolicity then becomes more direct, without needing to rely on traditional notions like curves or topology.  We also note that, from the physical point of view,
 the chronological relation has more operational meaning than the topology.
We conclude that, although from a traditional technical perspective this definition relies on the manifold topology,
 from a broader perspective it is topology-independent.

This is what makes the study of low-regularity spacetime geometry so compelling: one does not begin with predefined ideas of what the fundamental objects should be. Instead, it is the simplicity of the relationships among various objects that suggests which concepts should be taken as primitive. 

In previous work, we argued that one of the best definitions of global hyperbolicity is one that frames it within the context of topological ordered spaces \cite[Sec.\ 1.1, Def.\ 3.1]{minguzzi12d} \cite{minguzzi17,minguzzi23}: the causal relation is a closed order, and the causally convex hull operation preserves compactness. This definition is highly robust when it comes to weakening the regularity conditions of the spacetime, precisely because it depends only on the topology and order. Notably, the characterization in Theorem \ref{prtr} takes a completely different direction. It does not require a predefined topology, and the relation used is the  chronological one, not the  causal one.

Clearly, Theorem \ref{prtr} suggests using this characterization to define global hyperbolicity for sets endowed with a strict partial order (irreflexive, transitive, asymmetric) $I$, a mathematical setting quite distinct from that of closed ordered spaces.

To these authors, it remains puzzling that the notion of global hyperbolicity can be framed in such general yet distinct frameworks. This may further underscore the importance of this concept, though the origin of these manifestations has yet to be clarified.

We now wish to emphasize the formal analogy between the characterization in \ref{prtr} and the Hopf-Rinow theorem in Riemannian geometry \cite[Thm.\ 4.1]{kobayashi63}, which, in particular, states: Let $(M,g)$ be a connected and smooth Riemannian manifold. It is a complete Riemannian space if and only if the closed and bounded subsets of $M$ are compact (Heine-Borel property).

Of lesser importance to us is the fact that the notion of a complete Riemannian space is related to the completeness of geodesics. This type of notion is clearly among the least robust under low regularity, as geodesics require at least a Lipschitz connection to be defined. Similarly, of little importance is the fact that it is equivalent, by the same Hopf-Rinow theorem, to the completeness of the metric space $(M,d)$, where $d$ is the metric. What is important for us is that the concept of a `complete Riemannian space' can be introduced using the Heine-Borel property.
That is, if we refer to $B(p,r):=\{q\in M : \ d(p,q)<r\}$ as a `ball', we can rephrase\\

\begin{theorem}
Let $(M,g)$ be  a connected and smooth Riemannian space. It is a complete Riemannian space iff the balls are relatively compact in the topology generated by the balls.\\
\end{theorem}

There is a clear parallelism between Theorem \ref{prtr} and this theorem, with the following replacements
\begin{align*}
 \textrm{Spacetime} &\leftrightarrow \textrm{Riemannian space} \\
  \textrm{Globally hyperbolic spacetime} &\leftrightarrow \textrm{Complete Riemannian space} \\
   \textrm{Chronological diamond} &\leftrightarrow \textrm{Ball} \\
   \textrm{Alexandrov topology} &\leftrightarrow \textrm{Metric topology}
\end{align*}
The main difference is the absence of the requirement for the Hausdorff property in the positive signature result. This is because the triangle inequality ensures it. Additionally, the triangle inequality guarantees that the distance $d$ is continuous.
These properties are lost in the Lorentzian signature. However, the triangle inequality ensures, as it does in the Lorentzian signature, that the balls generate the same topology whether considered as a basis or subbasis.

Theorem \ref{prtr} thus supports the idea that globally hyperbolic spacetimes are the analog,  in the Lorentzian signature, to complete Riemannian spaces. This correspondence can be extended further, with characterization (e) serving as an analog of the Hopf-Rinow theorem, as envisioned by Burtscher and Garc\'{\i}a-Heveling. Perhaps Theorem \ref{prtr} clarifies more than other results that the analog of balls in the Lorentzian signature are the {\em chronological diamonds}. Of course, the compactness of causal diamonds, which is part of the definition of global hyperbolicity, has long been viewed as analogous to the Heine-Borel property. However, only when this compactness condition is expressed in terms of chronological diamonds can we dispense with the manifold topology and establish a compelling correspondence between the two signatures.

\section{Lorentzian distance function} \label{cnpda}
Let us define the functions dependent on a choice of $r\in M$, $d_r=d(r,\cdot)$, and $d^r=d(\cdot,r)$.

In the remainder of this work we shall be interested in the problem of establishing properties of the Lorentzian distance function that might promote a generic spacetime to a globally hyperbolic one. Theorem \ref{prtr} already provides a solution to this problem since the chronological relation, and hence the chronological diamonds, are expressible through the Lorentzian distance
\[
I=d^{-1}((0,+\infty]), \qquad I^+(p)=d_p^{-1}((0,+\infty]), \qquad I^-(q)=(d^q)^{-1}((0,+\infty]).
\]
Still, the formulation of global hyperbolicity would be rather indirect. We would prefer a formulation phrased directly in terms of continuity properties of $d$.

To start with, we shall explore some causality properties expressed via the Lorentzian distance function.\\

\begin{proposition} \label{cnqxz}
On a spacetime $(M,g)$ for any two events $p,q\in M$, the inclusion $I^+(q)\subset I^+(p)$ is equivalent to the property $d_q\le d_p$.
Dually, the   inclusion $I^-(p)\subset I^-(q)$ is equivalent to the property $d^p\le d^q$.\\
\end{proposition}

This result is a bit surprising since the function $d$  contains much more information than the chronological relation.
Note that we do not assume that $d$ is finite or continuous.

\begin{proof}
Suppose that  $I^+(q)\subset I^+(p)$  and let $r \in M$. If $r\notin I^+(q)$ then $d_q(r)=0$ and the inequality we wish to prove follows. If $r\in I^+(q)$ for any $m>0$, $m<d_q(r)$, there is a timelike curve $\gamma:[0,1]\to M$ connecting $q=\gamma(0)$ to $r=\gamma(1)$ of Lorentzian length $\ell(\gamma)> m$. Let $\delta>0$ be such that  $\ell(\gamma\vert_{[\delta,1]})> m$. But $\gamma(\delta)\gg q$ hence $\gamma(\delta)\gg p$, which implies $d_p(r)\ge  \ell(\gamma\vert_{[\delta,1]})> m$. Since $$m<d_q(r)$$ is arbitrary, $d_p(r)\ge d_q(r)$. By the arbitrariness of $r$,  $d_q\le d_p$.

For the converse, suppose that  $d_q\le d_p$, then if $r\in I^+(q)$ we have $d_q(r)>0$ so $d_p(r)>0$ which implies $r\in I^+(p)$. By the arbitrariness of $r$, $I^+(q)\subset I^+(p)$.
\end{proof}

We recall \cite[Def.\ 4.46]{minguzzi18b} that a spacetime is {\em future distinguishing} if $I^+(p)=I^+(q)\Rightarrow p=q$ (and similarly in the past case). A  spacetime is {\em weakly distinguishing} if $I^+(p)=I^+(q)$ and $ I^-(p)=I^-(q)\Rightarrow p=q$, cf.\ \cite[Def.\ 4.47]{minguzzi18b}. A  spacetime is {\em distinguishing} if $I^+(p)=I^+(q)$ or $ I^-(p)=I^-(q)\Rightarrow p=q$, cf.\ \cite[Def.\ 4.59]{minguzzi18b}.

It is convenient to introduce the following notions\\

\begin{definition}
We say that the {\em the Lorentzian distance future distinguishes events} (the spacetime is future $d$-distinguishing) if  $d_p= d_q \Rightarrow p=q$, and that is {\em past  distinguishes events} if $d^p= d^q \Rightarrow p=q$.  The spacetime is {\em  $d$-distinguishing} if it is both future and past $d$-distinguishing, namely:  $d^p= d^q$ or $d_p= d_q \Rightarrow p=q$. We say that the spacetime is {\em weakly $d$-distinguishing} if  $d^p= d^q$ and $d_p= d_q \Rightarrow p=q$.\\
\end{definition}

Note that $d$-distinction implies `future or past $d$-distinction' which implies weak $d$-distinction.\\

\begin{remark}
 Weak $d$-distinction is called just $d$-distinction (or it is said that the Lorentzian distance distinguishes events) in \cite{minguzzi22,minguzzi24b} as this is the only version that plays a role in those works.\\
 \end{remark}

An immediate consequence of Prop.\  \ref{cnqxz} is\\

\begin{corollary} \label{moprs}
On a spacetime $(M,g)$ (future/past/weak) $d$-distinction is equivalent to  (resp.\ future/past/weak)  distinction.\\
\end{corollary}
 One might ask why introducing new terminology of `$d$-' type if the equivalence holds. The reason is that for less regular spacetimes the equivalence might not hold anymore. This problem will be explored later on.
%
%
%
%
%

We recall \cite[Sec.\ 4.1]{minguzzi18b} that a spacetime is {\em future reflecting} if $I^-(p)\subset I^-(q) \Rightarrow I^+(p)\supset I^+(q)$. The reverse direction of the implication holds for the {\em past reflecting} case. A spacetime is {\em reflecting} if it is both past and future reflecting, that is if  $I^-(p)\subset I^-(q) \Leftrightarrow I^+(p)\supset I^+(q)$.\\

\begin{definition}
A spacetime is {\em future $d$-reflecting} if $d^p \le d^q \Rightarrow d_q\le d_p$.  A spacetime is {\em past $d$-reflecting} if $d^p \le d^q \Leftarrow d_q\le d_p$.   A spacetime is $d$-reflecting if it is both past and future $d$-reflecting, that is if  $d^p \le d^q \Leftrightarrow d_q\le d_p$.\\
\end{definition}

Another immediate consequence of Prop.\  \ref{cnqxz} is\\

\begin{corollary} \label{mopq}
On a spacetime $(M,g)$ (future/past) $d$-reflectivity is equivalent to  (resp.\ future/past)  reflectivity.\\
\end{corollary}
Note that the validity of ($d$-)reflectivity makes all the weak/past/future ($d$-)distinction properties coincide.


By using the Lorentzian distance we can characterize causal continuity.\\

\begin{theorem} \label{vkqpd}
A spacetime $(M,g)$ is causally continuous iff it is weakly $d$-distinguishing and $d$-reflective.\\
\end{theorem}

\begin{proof}
This follows from the improved characterization of causal continuity: weak distinction and reflectivity \cite{budic74} \cite[Def.\ 4.109]{minguzzi18b}.
\end{proof}

The following result is classical \cite[Thm.\ 4.24]{beem96} \cite[Prop.\ 5.2]{minguzzi18b} but is phrased here in a stronger form than usual.

In statements concerning continuity of the functions $d, d_p, d^p, p\in M$, it is understood that such property is with respect to the manifold topology of $M$ unless stated otherwise.\\

\begin{proposition} \label{nnrt}
Let $T$ be a topology on $M$
such that $p\in \overline{I^+(p)}\cap \overline{I^-(p)}$ for every $p\in M$ (e.g.\ the Alexandrov or the manifold topology). The upper $T$-semi-continuity of the functions $d^r$, $r\in M$, where they vanish implies future reflectivity. Similarly, the upper $T$-semi-continuity of the functions $d_r$, $r\in M$, where they vanish implies past reflectivity.\\
\end{proposition}

Note that in the statement the closure is with respect to $T$ and that in the proof we avoid use of sequences. Indeed, the first countability of $T$ is not assumed. The condition  $p\in \overline{I^+(p)}\cap \overline{I^-(p)}$,  means that the topology $T$ is fine but not too fine.

\begin{proof}
Let us prove the latter statement, the former being analogous. Suppose that past reflectivity does not hold, then we can find $p,q\in M$ such that $I^+(q)\subset I^+(p)$ but $I^-(p)$ is not a subset of $I^-(q)$, that is, there is $r\ll p$ such that $r\notin I^-(q)$. But $q\in \overline{I^+(q)}$ thus $q\in \overline{I^+(p)}\subset \overline{I^+(r)}$. For every $T$-open set $O$ of $q$ there is $q'\in O\cap I^+(p)$ and
\[
d(r,q')\ge d(r,p)+d(p,q')\ge d(r,p)>0
\]
which proves that $d_r$ has a discontinuity at $q$ where it vanishes.
\end{proof}

\begin{remark}
The condition ``$\forall p,q\in M$, $d(p,q)=\textrm{inf}_{q'\in I^+(q)} d(p,q')$ and  $\forall p,q\in M$, $d(p,q)=\textrm{inf}_{p'\in I^-(p)} d(p',q)$'' easily implies that $d^p, d_p, p\in M$ are upper semi-continuous and hence continuous in the Alexandrov topology. Hence under this condition reflectivity holds.\\
\end{remark}

Proposition \ref{nnrt} allows us to replace reflectivity in a number of results, for instance\\

\begin{theorem} \label{vkqpy}
Every weakly $d$-distinguishing spacetime for which $d$ is continuous on $d^{-1}(0)$ is casually continuous.\\
\end{theorem}

Let us compare the continuity of $d$ with that of the functions $d_p,d^p, p\in M$.

We already recalled that in the Alexandrov topology sets of the form $I(p,q)$ form a basis, not just a subbasis. Also sets of the form $I^+(p)$ and $I^-(q)$ are open in the Alexandrov topology.\\

\begin{proposition} \label{bpmf}
The functions $d$, $d_p$ and $d^p$, $p\in M$, are lower semi-continuous for the Alexandrov topology and hence any finer topology (e.g.\ the manifold topology). Conversely, if the functions $d_p$ and $d^p$, $p\in M$, are  lower semi-continuous for a topology then that topology contains the Alexandrov topology.\\
\end{proposition}

In other words,  the lower semi-continuity property is equivalent to  the property that the topology contains the Alexandrov topology.

\begin{proof}
The proof of the first statement goes as the usual one for the lower semi-continuity of $d$, e.g.\ \cite[Thm.\ 2.32]{minguzzi18b}.
For the second statement, lower semi-continuity of $d_p$ implies that the set $d_p^{-1}((0,+\infty])=I^+(p)$ is open, and lower semi-continuity of $d^q$ implies that the set $(d^q)^{-1}((0,+\infty])=I^-(q)$ is open, hence the sets of the form $I(p,q)$ are open.
\end{proof}

\begin{remark}
We are not concerned, in this section, with issues connected to low regularity since, as stated in the introduction, our spacetime $(M,g)$ is sufficiently regular. In any case, it is useful to recall  that some regularity is needed to ensure the lower semi-continuity of $d$ in the manifold topology, for instance (in the broad framework of  proper Lorentz-Finsler spaces) local Lipschitzness \cite[Thm.\ 2.53]{minguzzi17}, or at least, as the proof of that result shows, some no bubbling/causal space/push up condition  suffices  (see \cite[Thm.\ 2.8]{minguzzi17} for the equivalence). Without such conditions the lower semi-continuity  property does not hold but can be recovered in a $C^0$ framework modifying the notion of causal curve and Lorentzian distance \cite{ling24}.\\
\end{remark}

\begin{theorem} \label{clpqo}
If the Lorentzian distance $d$ is $T\times T$-continuous then the functions $d_p:=d(p,\cdot)$ and $d^p:=d(\cdot, p)$ are $T$-continuous for every $p\in M$. The converse holds if $T$ is a topology  such that $p\in \overline{I^+(p)}\cap \overline{I^-(p)}$  for every $p\in M$,
e.g.\ the manifold topology or the Alexandrov topology.\\
\end{theorem}

Once again the closure in this statement is with respect to $T$.


\begin{proof}
Note that whatever the direction considered, the assumption implies that $T$ is finer than the Alexandrov topology (because the $T$-continuity of $d_p, d^p$ implies the $T$-openness of $I^+(p), I^-(p)$).
Suppose that $d$ is $T\times T$-continuous. The function $\phi_p:M\to M\times M$, $q\to (p,q)$ is $T$-continuous, so is the function $d_p=d\circ \phi_p$. Similarly, $d^p$ is $T$-continuous.

For the converse, we already know by Prop.\ \ref{bpmf} that $d$ is lower semi-continuous with respect to $T\times T$  so it is sufficient to prove upper semi-continuity  with respect to $T\times T$  . Let $(p,q)$ be such that $d(p,q)<\infty$ (otherwise upper semi-continuity at $(p,q)$ is clear), and let $\epsilon>0$.  Let $p'\ll p$. By upper semi-continuity of $d^q$ and the condition $p\in \overline{I^-(p)}$ the point $p'$ can be chosen so close to $p$ that $d(p',q)<d(p,q)+\epsilon/2$. Let $q'\gg q$. By upper semi-continuity of $d_{p'}$ and the condition $q\in \overline{I^+(q)}$ the point $q'$ can be chosen so close to $q$ that $d(p',q')< d(p',q)+\epsilon/2$, thus $d(p',q')< d(p,q)+\epsilon$. The open sets $I^+(p')\ni p$, $I^-(q')\ni q$ belong to the Alexandrov topology and hence to the topology $T$; we have $(p,q)\in I^+(p')\times I^-(q')$ and on this product $T\times T$-open set  the Lorentzian distance is bounded by $d(p',q')$ and hence by $d(p,q)+\epsilon$, which proves the upper semi-continuity of $d$ with respect to $T\times T$.
\end{proof}

In the next result the topology involved is the manifold topology.\\
\begin{corollary} \label{coer}
The Lorentzian distance $d$ is continuous  iff the functions $d_p,d^p, p\in M$, are all continuous.\\
\end{corollary}


The key point is that by using only the Lorentzian distance function, we have been able to derive properties as robust as strong causality, as demonstrated in Theorem \ref{vkqpd}. This is significant because it establishes the Hausdorff property of the Alexandrov topology, thereby ensuring its equivalence with the manifold topology \cite[Thm.\ 4.75]{minguzzi18b}.


We recall the following key result \cite[Prop.\ 1.6]{minguzzi22} \\

\begin{proposition} \label{jept}
Let $\mathcal{A}$ be a family of open subsets for a topology such that (i) $\mathcal{A}$ separates points and (ii) every point admits an open neighborhood
in $\mathcal{A}$ contained in a compact set. Then $\mathcal{A}$ is a subbasis for the topology.\\
\end{proposition}

Note that by (i) and (ii) the topology is Hausdorff and locally compact. Property (i) can be called {\em $\mathcal{A}$-Hausdorffness}, while property (ii) can be called {\em $\mathcal{A}$-local compactness}.

This leads to the following result which  is interesting as it characterizes\footnote{This is still somewhat indirect in option (c) which can be deduced from the Hausdorffness of the Alexandrov topology where chronological diamonds can be written using $d$.} global hyperbolicity using only properties of the Lorentzian distance functions and without invoking the manifold topology.\\

\begin{theorem} \label{cjpr}
Let  $(M,g)$ be a spacetime which is weakly $d$-distinguishing, such that the chronological diamonds are relatively compact in a topology $T$ that makes the functions $\{d_p,d^p, p\in M\}$ lower semi-continuous, and such that at least one of the following  conditions hold
\begin{itemize}
\item[(a)]  $d$-reflectivity,
\item[(b)] the functions $\{d_p,d^p, p\in M\}$  are actually $T$-continuous and $p\in \overline{I^+(p)}\cap \overline{I^-(p)}$ for every $p\in M$, where closure is in the topology $T$,
\item[(c)] strong causality.
\end{itemize}
Then the topology $T$, the Alexandrov topology, the manifold topology and the initial topology of the family $\{d_p,d^p, p\in M\}$ are all coincident. Moreover, $(M,g)$ is globally hyperbolic and so $d$ is finite and continuous (in the corresponding product topology).\\
\end{theorem}

We recall that ``in a topology $T$ that makes the functions $\{d_p,d^p, p\in M\}$ lower semi-continuous'' is equivalent to  ``in a topology $T$ finer than the Alexandrov topology'' by Prop.\ \ref{bpmf}.

The role of (a),(b), and (c) is that of ensuring that $T$ is not too fine. We do not know if they can be dispensed of.

\begin{proof}
Let $\mathcal{A}$ be the family of chronological diamonds. Clearly, $\mathcal{A} \subset T$. The family  $\mathcal{A}$ provides a basis for  the Alexandrov topology. If we assume $d$-reflectivity (option (a)) then, from Thm.\ \ref{vkqpd}, the spacetime is causally continuous hence strongly causal which implies that the Alexandrov topology is Hausdorff \cite[Thm.\ 4.75]{minguzzi18b} and hence $\mathcal{A}$ separates points. Option (b) gives  $d$-reflectivity due to Prop.\ \ref{nnrt} and so we are back to the previous case.

Assuming strong casuality (option (c)) also implies that the  Alexandrov topology coincides with the manifold topology and so that $\mathcal{A}$ separates points.

 Every point admits as neighborhood a chronological diamond, which thus belongs to $\mathcal{A}$, and which is contained in a $T$-compact set.

 By Prop.\ \ref{jept},  $T$ coincides with the Alexandrov topology and hence (by the proved strong causality) with the manifold topology.

 The chronological diamonds are thus relatively compact in the manifold topology. A spacetime is globally hyperbolic iff it is strongly causal and every chronological diamond is precompact with respect to the manifold topology \cite[Cor.\ 3.3]{minguzzi08e}, thus $(M,g)$ is globally hyperbolic. In particular $d$, and $\{d_p,d^p, p\in M\}$ are finite and continuous in the manifold topology.

 Note that the manifold topology is finer  than the initial topology of the family  $\{d_p,d^p, p\in M\}$ which is finer than the Alexandrov topology. Since the first and third (last) topology coincide, we conclude that the manifold topology coincides with the initial topology of the family  $\{d_p,d^p, p\in M\}$.
\end{proof}

By using the manifold topology Theorem \ref{cjpr} gives the following corollary, as the continuity of $d$ implies $d$-reflectivity and the continuity of the functions $\{d_p,d^p, p\in M\}$.\\

\begin{corollary} \label{cfpr}
Let  $(M,g)$ be a spacetime which is weakly $d$-distinguishing, such that $d$ is continuous and such that the chronological diamonds are relatively compact. Then  the Alexandrov topology, the manifold topology and the initial topology of the family $\{d_p,d^p, p\in M\}$ are all coincident. Moreover, $(M,g)$ is globally hyperbolic and so $d$ is finite.\\
\end{corollary}

We recall the following definition \cite[Def.\ 2.1, Lemma 2.3, Thm.\ 2.14]{minguzzi24b} \\

\begin{definition} \label{mxpq}
A {\em Lorentzian metric space without chronological boundary} is a pair $(X,d)$, where $X$ is a set, $d: X\times X\to [0,\infty)$  satisfies the reverse triangle inequality for chronologically related triples $x\ll y \ll z$ (the chronological relation being $I=\{d>0\}$), $(X,d)$ is weakly $d$-distinguishing\footnote{In \cite{minguzzi24b} this is referred, more simply, as `$d$ distinguishes the events'.}, every point is contained in some chronological diamond,  and every chronological diamond is relatively compact in some topology $T$ that makes the function $d$ continuous. \\
\end{definition}

\begin{remark}
In what follows we shall call these objects simply `Lorentzian metric spaces' as we shall only consider the case of empty chronological boundary. In other words, each event is contained in some chronological diamond.\\
\end{remark}

The causal relation for a Lorentzian metric space is defined as follows \cite[Def.\ 4.1]{minguzzi24b}
\[
\tilde J=\{(p,q): d_p \ge d_q  \ \textrm{and} \ d^p\le d^q\},
\]
it is reflexive, transitive, antisymmetric and it satisfies $I\subset \tilde J$,  and $I\circ \tilde J\cup \tilde J \circ I \subset I$, where $\circ$ is the composition of relations. The reverse triangle inequality holds for triples of points that are $\tilde J$-related \cite[Thm.\ 4.4]{minguzzi24b}.

A great deal is known about Lorentzian metric spaces and the topology $T$ mentioned in the definition. In particular, the topology $T$ is unique and coincides with the initial topology of the functions $\{d_p, d^p, p\in M\}$, and so is referred as `the topology of the Lorentzian metric space'.
The causal relation $\tilde J$ is thus closed in the product topology $T\times T$.

On the regular spacetime $(M,g)$ let us introduce the relation \cite{minguzzi07e} \cite[p.\ 105]{minguzzi18b}
\begin{align*}
D&=\{(p,q): I^+(p)\supset I^+(q)\} \cap  \{(p,q): I^-(p)\subset I^-(q)\}
\end{align*}
then, by  Prop.\ \ref{cnqxz}, it can be rewritten
\begin{align}
D&= \{(p,q): d_p \ge d_q  \ \textrm{and} \ d^p\le d^q\}.
\end{align}
We see that when $(M,g)$ is a Lorentzian metric space, $\tilde J=D$.  Reflectivity is equivalent to $D=\bar J$, see \cite[Def.\ 4.9]{minguzzi18b}, and in this case
 \[
 D= \{(p,q): d_p \ge d_q \}=\{(p,q): d^p \le d^q \}.
 \]
  If the causal relation is closed we have $D=J$, see \cite[Thm.\ 4.13]{minguzzi18b}.

We  also know that every regular globally hyperbolic spacetime $(M,g)$ is a Lorentzian metric space  \cite{minguzzi24b}. The rationale for the introduction of Lorentzian metric spaces is that they are very general. For instance, the definition comprises discrete versions of spacetime such as causal sets \cite{surya19}.

Our previous Theorem \ref{cjpr} aims to answer the following question: under what conditions a regular spacetime $(M,g)$ such that $(M,d)$ is a Lorentzian metric space is globally hyperbolic? The main difficulty is in establishing that the Lorentzian metric space topology coincides with the manifold topology. In a sense, the real problem is that of justifying and recovering the manifold topology out of ingredients that do not use it.
For that we needed to impose properties (a), (b) or (c).

Thus, Theorem \ref{cjpr}  implies the following\\

\begin{theorem} \label{mxdf}
On a spacetime $(M,g)$  let $d$ be the associated Lorentzian distance. Suppose that  $(X,d)$ is a  Lorentzian metric space in the sense of Def.\ \ref{mxpq} which satisfies (a), (b) or (c) of Thm.\ \ref{cjpr}, then $(M,g)$ is globally hyperbolic, the Lorentzian metric space topology coincides with the manifold topology and $\tilde J=J$.\\
\end{theorem}

\begin{proof}
As recalled (\cite[Thm.\ 2.14]{minguzzi24b}) the Lorentzian metric space property implies that weak $d$-distinction holds  and every chronological diamond is relatively compact in some topology $T$ that makes the function $d$ continuous. In particular, the functions $\{d_p,d^p, \, p\in M\}$ are continuous hence lower semi-continuous. Assuming (a), (b) or (c) we can apply Thm.\ \ref{cjpr} and so infer global hyperbolicity and the equality for the topologies.  Under global hyperbolicity $J$ is closed thus $\tilde J=D=J$.
\end{proof}

This result suggests that in an abstract setting Lorentzian metric spaces augmented by properties (a), (b) or (c) could be particularly interesting. Property (b) is suited for Lorentzian length spaces in which any two chronologically related points are connected by an isochronal curve, however, it is too strong for causal sets. \\

\begin{corollary} \label{mxdf2}
On a spacetime $(M,g)$  let $d$ be the associated Lorentzian distance. Suppose that  $(X,d)$ is a  Lorentzian metric space in the sense of Def.\ \ref{mxpq} where for every $p\in M$, $p\in \overline{I^+(p)}\cap \overline{I^-(p)}$, where closure is with respect to the Lorentzian metric space topology. Then $(M,g)$ is globally hyperbolic, the Lorentzian metric space topology coincides with the manifold topology and $\tilde J=J$.\\
\end{corollary}

As mentioned, the assumption is verified for a subclass of  Lorentzian length spaces \cite[Def.\ 5.1]{minguzzi24b}, namely those for which the maximal connecting curves do not admit `null segments' in the sense that they are isochronal. In fact $p$ has some point $q\gg p$, then the isochronal curve connecting $p$ to $q$ is continuous in the Lorentzian metric space topology which implies  $p\in \overline{I^+(p)}$ and similarly in the past case.
\begin{proof}
It follows from option (b) of Thm.\ \ref{cjpr}.
\end{proof}

There  are examples of Lorentzian metric spaces for which the natural (Lorentzian metric space) topology is the discrete one, so the property $p \in \overline{I^\pm(p)}$ does not hold. In those more general cases a property such as  (a) would be  preferred.

We single out the following results which characterizes globally hyperbolic spacetimes by using solely the Lorentzian distance.\\

\begin{theorem} \label{refpo}
A regular spacetime $(M,g)$ is globally hyperbolic if and only if it is a $d$-reflective Lorentzian metric space. In this case the Lorentzian metric space topology coincides with the manifold topology and $\tilde J=J$.\\
\end{theorem}

From Cor.\ \ref{mopq} and Prop.\ \ref{nnrt} we also get\\

\begin{theorem} \label{refpo2}
Let $(M,g)$ be a regular spacetime such that its Lorentzian distance $d$ is continuous in the manifold topology. It is globally hyperbolic if and only if it is a  Lorentzian metric space. In this case the Lorentzian metric space topology coincides with the manifold topology and $\tilde J=J$.\\
\end{theorem}

We are now going to improve Theorem \ref{refpo} further.\\

\begin{theorem} \label{coox}
Let $(M,g)$ be a regular spacetime.  It is globally hyperbolic if and only if it is a future $d$-distinguishing or past $d$-distinguishing  Lorentzian metric space.  Moreover, in this case the Lorentzian metric space topology coincides with the manifold topology and $\tilde J=J$.\\
\end{theorem}

It is equivalent to ask that $(M,g)$ is a `future or past distinguishing' spacetime and $(M,d)$ is a Lorentzian metric space.

This is an improvement of Thm.\ \ref{refpo} because $d$-reflectivity improves the weak $d$-distinction property implicit in the definition of Lorentzian metric space to $d$-distinction.

\begin{proof}
The direction to the right was proved in \cite[Prop.\ 2.4, Thm.\ 2.14]{minguzzi24b}

For the converse, we want to prove that $d$-reflectivity holds and then apply Thm.\ \ref{cjpr}.

Since the manifold can be covered by a countable family of charts, by selecting the points corresponding to rational coordinates, we get a countable family of points $\mathcal{S}$ such that for every $p\in M$, there are $x,y\in \mathcal{S}$ such that $p\in I(x,y)$. As a consequence, $(M,d)$ is a countably generated Lorentzian metric space \cite[Def.\ 2.12, Def.\ 3.17]{minguzzi24b} which implies that the Lorentzian metric space topology is Polish \cite[Prop.\ 3.20]{minguzzi24b}. In particular, being metrizable, compactness for this topology is equivalent to sequential compactness.

Recalling the definition of causal relation for a Lorentzian metric space, namely $\tilde J=\{(p,q): d_p \ge d_q  \ \textrm{and} \ d^p\le d^q\}$, we get $\tilde J=D$. Observe that $I=\{d>0\}$ thus $I$ is also the chronological relation for $(M,d)$.

Let us prove past $d$-reflectivity. Let $p,q\in M$ and suppose that $d_p\ge d_q$. Let $\gamma:[0,1]\to M$, $\gamma(0)=q$, be a smooth timelike curve. Then for each $t>0$, $d_q(\gamma(t))>0$ which implies $d_p(\gamma(t))>0$ and hence $p \ll \gamma(t)$. The chronological diamond $I(p,\gamma(1))$ is relatively compact in the Lorentzian metric space topology. The sequence $q_n:=\gamma(1/n)$ must have a subsequence $x_k=q_{n_k}$ which converges to  some point $x$ in the Lorentzian metric space topology (note that we do not assume the continuity of the curve in the Lorentzian metric space topology), see Fig.\ \ref{figpr}.

Note that $(p,x_k)\in I \subset \tilde J$ thus $(p,x)\in \tilde J=D$.  If $x=q$  we have $(p,q) \in \tilde J$ and, by the definition of $\tilde J$, we have
$d^p\le d^q$, and we have finished.


There remains the case in which $x\ne q$.
We want to show that this possibility does not apply as it leads to contradictions.

Since for each $t>0$ the sequence $x_k$ is contained (for sufficiently large $k$)  in $\tilde J^{-}(\gamma(t))$, which is closed in Lorentzian metric space topology, the point $x$ belongs to $\cap_{t>0} \tilde J^{-}(\gamma(t))$. However, by $I\circ \tilde J\cup \tilde J \circ I \subset I$, for $t'>t$, $\tilde J^{-}(\gamma(t)) \subset I^-(\gamma(t'))$. Now, for every $t'>0$ we can find $t$, $0<t<t'$, thus $x\in \cap_{t'>0}  I^-(\gamma(t'))$ and relabelling
$x\in \cap_{t>0}  I^-(\gamma(t))$. This implies $q\in \overline{I^+(x)}$ (closure in the manifold topology) or equivalently $I^+(q)\subset I^+(x)$. However, the sequence $x_k$ is also contained in $\tilde J^{+}(q)$ which is  closed in Lorentzian metric space topology thus $x$ belongs to it and hence $(q,x)\in D$ by the proved equality $D=\tilde J$.
Since  $q\in \overline{I^+(x)}$  and  $x\in \overline{I^+(q)}$, we have by the openness of $I$, $I^+(x)=I^+(q)$ (see e.g.\ \cite[Eq.\ (4.1)]{minguzzi18b}).



We recall that the weak $d$-distinction property inherent in the Lorentzian metric space definition implies the weak distinction property.


Consider the sequence $x_{k}\gg x_{k+1}$, $x_k\in I^+(q)=I^+(x)$, $x_k\to x$ in the Lorentzian metric space topology, $x_k\to q$ in the manifold topology. Let $\sigma_k: [0,1]\to M$ be a timelike curve connecting $x$ to $x_k$. They do not contract to a point because in the manifold topology $x_k\to q \ne x$.  But $\sigma_{k}((0,1))\subset I(x,x_k)$, thus $\sigma_{k}\subset D(x,x_k)=\tilde J(x,x_k)$ which are compact in the Lorentzian metric space topology (and satisfy the finite intersection property). Note that $\cap_k \tilde J(x,x_k)=\{x\}$ because if $y\in \tilde J^-(x_k)\cap \tilde J^+(x)$ for every $k$, then, $(y,x_k)\in \tilde J$ implies  $(y,x) \in \tilde J$ by the closure of $\tilde J$ in the Lorentzian metric space  topology, and by its antisymmetry $(y,x)\in \tilde J$ and $(x,y)\in \tilde J$ implies $y=x$.

 For every neighborhood (in the  Lorentzian metric space topology) $O\ni x$ there is a sufficiently large $k$ such that $\sigma_k \subset D(x,x_k) = \tilde J(x,x_k)\subset O$. However, the limit curve theorem (version for regular spacetimes), applied to the sequence $\sigma_k$, implies that there is either: (i) a causal curve connecting $x$ to $q$, but this implies  $(x,q)\in J\subset D$, which implies (we proved above $(q,x)\in D$) that $D$ is not antisymmetric, a violation of weak distinction, a contradiction; or (ii) there are a past-inextendible causal curve $\sigma^q: (-1,0]\to M$ of ending point $q$ and a future inextendible causal curve $\sigma_x:[0,1)\to M$ of starting point $x$, both suitable limits of $\sigma_k$.



\begin{figure}[ht]
\centering
\includegraphics[width=8cm]{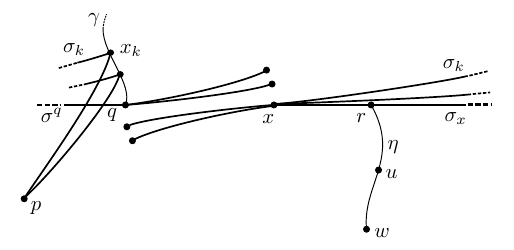}
\caption{A figure illustrating the proof of Thm.\ \ref{coox} under the assumption $q\ne x$. We have $(q,x)\in D$, $I^+(q)=I^+(x)$, $I^-(x)=I^-(r)$, so past and future distinction are violated. The figure suggests $I^+(r)\cap \sigma_k\ne \emptyset$ for any sufficiently large $k$, but this is just due to its bidimensionality.} \label{figpr}
\end{figure}

Now, consider  the curve $\sigma_x$ and let $r=\sigma_x(s)$, for some $s\in (0,1)$. Since $x\le r$ we know that $I^+(x)\supset I^+(r)$ and $I^-(x)\subset I^-(r)$. If $I^-(x)\ne I^-(r)$ there is $w\ll r$ such that $x \notin I^+(w)$. Let us consider a timelike curve $\eta:[0,1]\to M$ connecting $w$ to $r$, then we can find $u=\eta(1/2)$ so that $d(w,u)=c>0$, $d(u,r)>0$. Then, as $d(w,x)=0$, $x\in d_w^{-1}((-\infty,c))$ and $r \in I^+(u)$ where $d_w^{-1}((-\infty,c))\cap I^+(u)=\emptyset$. But $O:=d_w^{-1}((-\infty,c))$ is a Lorentzian metric space topology open neighborhood of $x$ thus, as we have shown above, $\sigma_k$ is contained in it for sufficiently large $k$ and so it cannot intersect $I^+(u)$. But as $r\in I^+(u)$ and $r\in \sigma_x$, this contradicts convergence in the $C^0$ topology (induced on the space of curves by the manifold topology \cite{hawking73}) of $\sigma_k$ to $\sigma_x$ which holds (actually a  stronger form of uniform convergence for suitably parametrized curves holds) for the standard limit curve theorem for regular spacetimes \cite{minguzzi18b}. The contradiction proves that $I^-(x)= I^-(r)$ (and by the arbitrariness of $s$, $I^-(x)= I^-(\sigma_x(s))$ for every $s\in (0,1)$). Now we have $I^+(q)=I^+(x)$, $q\ne x$, $I^-(x)=I^-(r)$, $x \ne r$, thus both future and past distinction are violated, which  contradicts the assumption.

Summarizing, the assumption $x=q$ leads to the proof of $d$-reflectivity, while the assumption $x\ne q$ leads to a contradiction and so  does not apply. Thus past $d$-reflectivity is proved.

Future $d$-reflectivity is proved similarly (using again the validity of   `past or future distinction'). We conclude, by Thm.\ \ref{mxdf}, that the spacetime is globally hyperbolic and the Lorentzian metric space topology coincides with the manifold topology. For a globally hyperbolic spacetime $J$ is closed, thus $D=J$ which implies $\tilde J=J$.
\end{proof}

\begin{remark}
It is natural to try to improve the theorem by dropping the assumption of future or past $d$-distinction. The assumption would then be `weak $d$-distinction' as it is contained in the definition of Lorentzian metric space. The proof would proceed as above reaching the case $q\ne x$ where we would need to push further the description of the spacetime in search of a contradiction.

In this respect it is worth nothing that the Lorentzian metric space topology looks rather peculiar. Indeed, since $I^+(q)=I^+(x)$, $I^-(q)\subset I^-(x)$, the two points are distinguished by a point $v\in I^-(x)$, $q\notin I^+(v)$. But then, chosen $b,c\in M$ such that $v\ll b\ll c\ll x$, since $I^+(v)\supset \tilde J^+(b)$, $q\in I^+(v)^C\subset \tilde J^+(b)^C=D^+(b)^C$ and $x\in I^+(c)$. Now, $I^+(c)\cap (\tilde J^+(b))^C=\emptyset$ where the two sets are open sets in the Lorentzian metric space topology (and the former also in the manifold topology). But any manifold topology open neighborhood of $q$ intersects $I^+(q)=I^+(x)\subset I^+(c)$ the Lorentzian metric space topology open neighborhood of $x$.

Another interesting observation is that for $t'>t$, $I^+(\sigma_x(t'))\subsetneq   I^+(\sigma_x(t))$, indeed if they were to coincide we would have a violation of weak distinction and we would have finished. Note that $(x, x_1)\in I$, thus for any sufficiently small $\epsilon>0$, $(\sigma_x(\epsilon), x_1)\in I$.
For given $\epsilon$ we must have for sufficiently large $k$, $(\sigma_x(\epsilon), x_k)\notin D$, otherwise $(\sigma_x(\epsilon), x) \in D$, which jointly with $(x,\sigma_x(\epsilon))\in J\subset D$ would again give a violation of weak distinction. This means that there is some $0<\tau=\inf \{s\}$ where $s$ are parameters such that $(\sigma_x(\epsilon), \gamma(s))\in I$.
In principle there could be an achronal lightlike geodesic connecting $\sigma_x(\epsilon)$ to $\gamma(\tau)$. We hope that this type of description could lead to a counterexample to the possibility of improving the theorem or to an improved assumption.
\end{remark}

\section{Globally hyperbolic Lorentz-Finsler spaces} \label{cnqp4}

In \cite{minguzzi17} we developed the causality theory for manifolds endowed with upper semi-continuous cone structures, see  \cite{bernard18,bernard19} for results on time functions. We already mentioned that for smooth spacetimes $(M,g)$, global hyperbolicity implies that $(M,d)$ is a Lorentzian metric space.  In this section we seek to obtain a similar result for these more  general structures.

In the following theorem and subsequent proof, the terminology of \cite{minguzzi17} is used. Unfortunately, it would be rather long to recall all the definitions here, so we shall assume the reader to be acquainted with that work.

We recall that a {\em Lorentzian length space} is a Lorentzian metric space for which any two causally related events are connected by a maximizing isocausal curve \cite[Def.\ 5.1]{minguzzi24b}.\\

\begin{theorem} \label{ckkq}
    Let $(M,C,\mathscr{F})$ be a globally hyperbolic locally Lipschitz proper Lorentz-Finsler space, where $C$ is the cone distribution and $\mathscr{F}$ the Finsler fundamental function. Assume that $\mathscr{F}(\partial C)=0$. Let $d$ be the associated Lorentz-Finsler distance function. Then $(M,d)$ is a Lorentzian length space, the Lorentzian metric space topology coincides with the manifold topology, and the chronological and causal relations for the Lorentzian metric space coincide with those of the Lorentz-Finsler space.\\
\end{theorem}
\begin{proof}
  To begin with, let us note that by \cite[Prop. 2.23]{minguzzi17}, since  $(M,C,\mathscr{F})$  is a  proper Lorentz-Finsler space, $(M,C)$ is a proper cone structure.

Originally, the chronological relation is  defined via piecewise $C^1$ timelike curves and not via the equality $I=\{d>0\}$. However, the inclusion $\supset$  holds thanks to  \cite[Thm.\ 2.56]{minguzzi17} (here we have used the fact that the fundamental function vanishes on the boundary), while the inclusion $\subset$ holds because on a proper Lorentz-Finsler space $\mathscr{F}$ is positive on $\textrm{Int}(C_p)$ and hence on $(\textrm{Int} C)_p$ (this follows from the concavity properties of $\mathscr{F}$, its non-negativity, and the fact that $C^\times_P:=\{(y,z): y\in C_p, \ \vert z\vert \le \mathscr{F}(y)\}$,  must be a proper cone i.e.\
 have non-empty interior).


    We have to verify the properties of the Lorentzian metric space of Def.\ \ref{mxpq},
    that is, (i) the reverse triangle inequality, (ii) in the manifold topology $d$  is continuous and the chronological diamonds are relatively compact (this would also show that the Lorentzian metric space topology is the manifold topology), (iii) weak $d$-distinction. The fact, (iv), that every event is contained in a chronological diamond is clear in the context of cone structures.

    The function $d$ satisfies the reverse triangle inequality \cite[Eq.\ (2.7)]{minguzzi17} and is lower semi-continuous  \cite[Thm.\ 2.53]{minguzzi17} in the manifold topology of $M$ {(here we have used again  the fact that the fundamental function vanishes on the boundary)}. By (g) and (d) of \cite[Thm.\ 2.60] {minguzzi17} it is also upper semi-continuous, thus continuous.

    For each $p,q\in M$, $I(p,q)\subset J(p,q)$ (here $J$ denotes the causal relation for $(M,C)$). By global hyperbolicity of $(M,C)$, $J(p,q)$ is compact. Taking into account Hausdorffness of the manifold topology,
    $\overline{I(p,q)}\subset \overline{J(p,q)}$ is also compact.

    We need only to prove property (iii).   By \cite[Thm.\ 2.47]{minguzzi17} (recall that by definition of Lorentz-Finsler spaces, $(M,C)$ is a closed cone structure, so the cited theorem applies), globally hyperbolic cone structure is distinguishing in the sense of \cite[Def.\ 2.16]{minguzzi17},
    so, by \cite[Prop.\ 2.18]{minguzzi17}\footnote{Note we can use here both options (a) and (b), since global hyperbolicity implies reflectivity, while proper and local Lipschitz properties are assumed in the statement.} $D_f=\{(p,q): q\in \overline{J^+(p)}\}$ is antisymmetric. But for $p,q\in M$, $p\ne q$, we have $q\in \overline{J^+(q)}$ and $p\in \overline{J^+(p)}$, so we cannot have   $\overline{J^{+}(p)}\neq \overline{J^{+}(q)}$ otherwise $q\in \overline{J^{+}(p)}$ and $p\in \overline{J^{+}(q)}$ which contradicts antisymmetry of $D_f$.
    As a consequence, due to \cite[Thm.\ 2.7]{minguzzi17} it cannot be ${I^{+}(p)}= {I^{+}(q)}$, which implies that the distance $d$ distinguishes the two events. This proves that $(M,d)$ is a Lorentzian metric space and its topology is the manifold topology.

    The equivalence of the chronological relations has been already proved. By \cite[Thm.\ 2.47]{minguzzi17} the causal relation $J$ of the cone structure is closed and by \cite[Thm.\ 2.7]{minguzzi17} $J= \bar I$. The causal relation for the Lorentzian metric space is $\tilde J= \{(p,q): d_p \ge d_q  \ \textrm{and} \ d^p\le d^q\} \subset  \{(p,q): I^+(q)\subset I^+(p) \ \textrm{and} \ I^-(p) \subset I^-(q)\} \subset \bar J=J$. But $\tilde J$ contains $I$ and is closed in the Lorentzian metric space topology and hence in the manifold topology $J=\bar I\subset \bar{\tilde J}=\tilde J$, thus $J=\tilde J$.

    Finally, the Lorentzian metric space is actually a Lorentzian length space thanks to the validity of the Avez-Seifert theorem \cite[Thm.\ 2.55]{minguzzi17}. Note that the maximizing curve $\sigma: [0,1]\to M$ for the Lorentz-Finsler space is absolutely continuous hence continuous, and since manifold topology and Lorentzian metric space topology coincide, they are also continuous in the latter topology. The identity $\ell(\sigma)=d(\sigma(0),\sigma(1))$, where $\ell$ is the Lorentzian length functional, implies, via the reverse triangle inequality, that for every triple of intermediate points $0\le t< t'< t''\le 1$, \[
    d(\sigma(t), \sigma(t'))+d(\sigma(t'),\sigma(t''))=d(\sigma(t),\sigma(t'')) ,
    \]
    which is the maximization property in the sense of Lorentzian length spaces.
    \end{proof}

    The maximizing causal curve connecting two causally related events need not be of definite causal character. For that a slightly more restrictive condition on  $\mathscr{F}$ must be imposed \cite[Thm.\ 3.1, 4.2]{minguzzi19b} (see \cite{graf18} for the Lorentzian case).
%
%

\section{Distance preservation implies Lorentzian isometry} \label{cmp3q}

We believe the result of this section is a kind of folklore results which, nevertheless, requires a detailed argument as a proof would also answer  recent questions in \cite{sakovich24}. It shows that under low regularity it is natural to focus on bijections that preserve the Lorentzian distance, as done e.g.\ in \cite{minguzzi22}, as under sufficient regularity they are equivalent to Lorentzian isometries.\\

\begin{lemma} \label{iqq}
On a distinguishing smooth spacetime $(M,g)$ for any future-directed timelike curve $\gamma : I\to M$ passing through $p\in M$, $p=\gamma(0)$, with tangent $v\in T_pM$,
\[
g_p(v,v)=-\frac{\dd^2}{\dd t^2} \frac{1}{2} d^2_p(\gamma(t))\vert_{t=0}
\]
\end{lemma}

Note that we are not claiming $g_p=\frac{1}{2} \textrm{Hess} \,d_p^2$  as this is not entirely true in Lorentzian geometry because $d_p$ vanishes and has zero Hessian outside the cone.

\begin{proof}
The point $p$ admits an arbitrarily small distinguishing neighborhood $U$, that is such that every causal curve with one endpoint in $p$ and the other in $U$ is necessarily contained in $U$. As a consequence, $d_p\vert_U$ coincides with the $d_p^U$, the Lorentzian distance from $p$ of the spacetime $(U, g\vert_U)$. In particular, $U$ can be chosen inside a convex neighborhood. Let $\varphi:U\to \mathbb{R}$ be the smooth function composition of the  smooth function $\exp^{-1}_p\vert_U: U\to T_p M$ with the smooth function $v\mapsto -g_p(v,v)$. Observe that $d_p^2 \vert_{I^+(p,U)}=\varphi\vert_{I^+(p,U)}$ which is smooth.


For every smooth curve $\gamma$ passing though $p=\gamma(0)$ with timelike tangent $v$, we have $\frac{\dd}{\dd t} \varphi(\gamma(t))\vert_{t=0}=0$ as $\varphi(p)=0$ which is the minimum, and so $\frac{\dd}{\dd t} d^2_p(\gamma(t))\vert_{t=0}=0$ and also, by the arbitrariness of $v$, $\dd \varphi=0$ at $p$.
Now observe that $\frac{\dd^2}{\dd t^2} \varphi\vert_{t=0}$ depends only on $\dot \gamma=v$ and not on $\ddot \gamma(0)$, thus any curve passing from $p$ with first derivative $v$ gives the same result, in particular a geodesic. For $v$ timelike it holds $d^2_p(\gamma(t))=-g_p(v,v)t^2$, thus for any curve with tangent $v$, $\frac{\dd^2}{\dd t^2} \frac{1}{2} d^2_p(\gamma(t))\vert_{t=0}=-g_p(v,v)$.
\end{proof}

\begin{theorem}
Let $(M,g)$  and $(M',g')$ be  strongly causal smooth spacetimes and let $f: M \to M'$ be a distance preserving bijection:
\[
d'(f(p),f(q))=d(p,q).
\]
 Then $M$ and $M'$  have the same dimension, $f$ is a (smooth) diffeomorphism such that $f_*$ sends at every point the future timelike cone into the future timelike cone, and $g=f^* g'$.\\
\end{theorem}

The proof of this result could have passed through Hawking et al.\ \cite{hawking76} or Malament's result \cite{malament77} to get first a conformal isometry, but we prefer to give a self contained proof, in which the Lorentzian distance, rather than the chronological/causal relations, plays a key role.

\begin{proof}
Let $p\in M$. By strong causality there is a chronologically convex neighborhood $U\ni p$. We choose $U$ so small that it is included in a convex neighborhood. Since $f$ preserves $d$, it sends chronologically related events to chronologically related events and the same holds for $f^{-1}$. As $U$ is chronologically convex, it is sent to a chronologically convex set $U'=f(U)$. Let $W$ be a convex neighborhood of $p'=f(p)$ then we can find $w',z'\in W$, $w'\ll p' \ll z'$, $I^+(w')\cap I^-(z')\subset W$, thus we can find $x,y\in U$, $w\ll x\ll p\ll y\ll z$. From here $I^+(x')\cap I^-(y') \subset W\cap  f(I^+(x)\cap I^-(y))$.
 Thus by redefining $U$ as $I^+(x)\cap I^-(y)$ we obtain that $U'$ becomes
 $I^+(x')\cap I^-(y')$ and so contained in a convex neighborhood $W$ of $p'$. 

By chronological convexity $d\vert_{U\times U}=d_U$, where $d_U$ is the Lorentzian distance function of $(U,g\vert_U)$, and $d_U$ is finite and continuous because $U$ is contained in a convex neighborhood. Let $t_i \in T_pM, i=1, \ldots, n+1$, with $n+1$ the dimension of the spacetime, be past timelike vectors that form a basis of $T_pM$, and let $q_i=\exp_p t_i$ where $t_i$ is rescaled, if needed, so that $q_i\in U$.
Since $U$ is contained in a convex neighborhood and $q_i\ll p$ there is only one timelike geodesic in $U$ connecting $q_i$ to $p$, and furthermore there cannot be pairs of conjugate points in a convex neighborhood so $q_i$ and $p$ are not conjugate.

This implies that $d_{q_i}=\Vert\exp_{q_i}^{-1} \Vert$, where $\Vert v \Vert=\sqrt{-g_{q_i}(v,v)}$ is the Lorentzian ``norm'', is smooth in a neighborhood of $p$. Its differential at $p$ is $g(u_i, \cdot)$ where $u_i:=\hat t_i$ is a normalized past-directed timelike vector, see e.g.\ \cite[Thm.\ 5]{minguzzi13d}.
The map $U\to \mathbb{R}^{n+1}$, $r \mapsto (d_{q_1}(r), \ldots, d_{q_{n+1}}(r))$ has thus non-singular Jacobian at $p$.

By the inverse function theorem the functions $\{d_{q_i}: i=1, \ldots, n+1\}$ provide in a neighborhood of $p$ the chart of a smooth atlas. Let  $q_i'=f(q_i)$, $i=1, \ldots, n+1$ and recall that $p'=f(p)$. We can construct analogous functions $d'_{q_i'}$ in a neighborhood of $p'$. We have already argued above that $U'$ is a neighborhood of $p'$ which contains all $q_i'$, and that $U'$ is contained in a convex neighborhood of $p'$.
The functions $d'_{q_i'}$ distinguish points because  otherwise the preimage of two undistinguishable points would not be distinguished by the functions $d_{q_i}$, a contradiction.
Thus we can construct this type of chart on a neighborhood of $p$ and $p'$ and with respect to these charts the map $f$ is the identity of some open set $O\subset \mathbb{R}^{n+1}$ (this follows from distance preservations $d_{q_i'}(f(r))=d_{q_i}(r)$). This shows that $M$ and $M'$ have the same dimension and $f$ is smooth (the function components are smooth in a chart belonging to a smooth atlas).

Now, since $f$ is a diffeomorphism $f_*:T_pM\to T_{p'} M'$ is linear and of maximal rank (hence open). If a future timelike vector is sent to a non
future timelike vector then there is some timelike vector $w$ sent to a non-zero non future causal vector $z$ (i.e.\ we can find $w$ so that its image $z$ is spacelike or causal but past-directed). But if $\gamma$ is a $C^1$ curve with tangent $w$ and $\gamma'=f \circ \gamma$, we have $z=f_* w$, and hence $z$ is the tangent to $\gamma'$. But $d_p^2(\gamma(t))=d_{p'}^2(\gamma'(t))$ which is a contradiction because the former is positive for small $t$ while the latter is zero for sufficiently small $t$, as $\gamma'(t)\notin I^+(p', U')$. This shows that $f_*$ sends the future timelike cone of $T_pM$ into the future timelike cone of $T_{p'} M'$ (it is a bijection, it is sufficient to consider $f_*^{-1}$).

Thus if $v$ is future timelike $v':=f_* v$ is future timelike. Let $\gamma : I\to M$ be a smooth curve passing through $p\in M$, $p=\gamma(0)$, with tangent $v\in T_pM$, then by Lemma \ref{iqq}, setting $\gamma'=f \circ \gamma$,
\[
g_p(v,v)=-\frac{\dd^2}{\dd t^2} \frac{1}{2} d^2_p(\gamma(t))\vert_{t=0}=-\frac{\dd^2}{\dd t^2} \frac{1}{2} d'{}^2_{p'}(\gamma'(t))\vert_{t=0}=g_{p'}(v',v'),
\]
which, by the polarization identity, is equivalent to $g_{p}=f^* g_{f(p)}$, as we desired to prove.
\end{proof}

\subsection{The Lorentz-Finsler case}

In this subsection we want to extend the previous result on Lorentzian geometry to the Lorentz-Finsler case.  We work with a pair $(M,L)$ were $M$ is a paracompact connected manifold, and the  \emph{Finsler Lagrangian} $L: TM\to \mathbb{R}$ is a $C^0$, positive  and homogeneous of degree two function. We also assume that $L$ is smooth on $TM\backslash 0$ and  that has vertical Hessian $g$ of Lorentzian signature there (the differentiability conditions on $TM\backslash 0$ can be weakened, we shall not be interested in the optimal conditions).  These assumptions imply that $L$ is actually $C^1$ \cite[p.\ 5]{minguzzi13d}.
Furthermore, its is assumed that it is possible to select  a continuous  set-valued function $C: M\to TM\backslash 0$, $x\mapsto C_x$ such that $C_x$ is a component of the set $\{v\in T_xM\backslash 0 :L(x,v)\le 0\}$. The subset $C_x\subset T_xM\backslash 0$ is termed {\em future} causal cone at $x\in M$. It is often convenient to treat $C$ as a cone subbundle of $TM\backslash 0$.

We shall also use $F: C\to [0,+\infty)$, defined by $L=-\frac{1}{2} F^2$. We extend $F$ by zero to the whole $TM$, so that $F(p, \cdot)$ becomes continuous on the whole $T_pM$.

We note that $(M,C,F)$ is   a Lorentz-Finsler space \cite[Sec.\ 2.13, 3.7]{minguzzi17} (the approach used here in Subsection \ref{cnqp4}), while $(M,L)$ is a Finsler spacetime  \cite{minguzzi13d}. In particular, we will use the local properties of  geodesics, normal neighborhoods and the exponential map from \cite{minguzzi13d} while using the concept of distinction defined for more general closed cone structures in \cite{minguzzi17}.


We start generalizing the Busemann-Meyer formula to the Lorentz-Finsler case.\\

\begin{lemma} \label{qmmg}
Let $f: U \to \mathbb{R}^k$, $U\subset \mathbb{R}^d$,   be a continuous function defined on an open set $U$ and differentiable in $U\backslash \{x_0\}$ for $x_0\in U$. If $D f(x)$ converges for $x\to x_0$ to some linear map $L: \mathbb{R}^d\to\mathbb{R}^k$ then $f$ is strongly differentiable at $x_0$ with strong differential $L$ at $x_0$.
As a consequence $f$ is  Lipschitz in a neighborhood of $x_0$. If $Df(x)$ is bounded (but not necessarily convergent) in neighborood of $x_0$, then $f$ is Lipschitz in  a neighborhood of $x_0$, too.\\
\end{lemma}

\begin{proof}
The first part is an improved statement of \cite[point (2), p.\ 970]{nijenhuis74} where the author assumes that $f$ is differentiable at $x_0$ and $L=Df(x_0)$. For $d=1$ it is well known, and consequence of L'Hôpital's rule, that if $f$ is differentiable in a neighborhood of $x_0$ and the derivatives converge to $L$, then   $f$ is differentiable at $x_0$ with derivative $L$.

So, let us assume $d\ge 2$ and let us go through the proof of \cite[point (2), p.\ 970]{nijenhuis74}. The Lagrange mean value theorem there invoked uses differentiability just on the interior of the segment $[x_1,x_2]$ there considered. So in the proof we just have to take $x_1,x_2$ such that the segment connecting them does not pass from $x_0$. Then we arrive at the equation for every pair of such points
\[
\vert f(x_2)-f(x_1)-L(x_2-x_1)\vert\le\epsilon \vert x_2-x_1\vert
\]
as in that proof. By continuity the inequality holds for every pair of points which proves that $f$ is strongly differentiable at $x_0$ with strong differential $L$.


Similarly, for the second part we can find a convex neighbourhood $U$ of $x_0$ such that for some $C>0$ we have $\vert Df\vert<C$ on $U\setminus
 \{x_0\}$. Then for $x_1,x_2\in U$ such that the interior of the segment connecting them does not contain $x_0$ we have
 \[
 \vert f(x_2)-f(x_1)\vert \leq C \vert x_2- x_1\vert,
 \]
 and, using continuity, we can conclude the proof along the same lines as for the first part.
\end{proof}

\begin{lemma}
Let a spray be smooth on $TM\backslash 0$. Then it is  $C^{1,1}$ in a neighborhood of the zero section with zero derivatives on the zero section.\\
\end{lemma}

In other words the local coefficients $G^\alpha(x,y)$ are $C^{1,1}$  in a neighborhood of the zero section $\{y=0\}$.\\

\begin{proof}
We use local coordinates on $TM$ induced by a chart on $M$.
Since $G^\alpha(x,y)$ is positive homogeneous of degree 2, the first derivatives are positive homogeneous of degree 2 or 1 depending on whether we differentiate with respect to $x^\beta$ or $y^\beta$. Let us introduce a complete Riemannian metric $h$ on $M$ and let us consider the unit sphere bundle of it.
Let us consider a compact neighborhood of $p$ and the $h$-unit sphere bundle projecting on it, which is thus compact. All the derivatives of the spray will be bounded on such bundle and so on the corresponding ball (minus the zero section) by positive homogeneity.  Now, define $G^\alpha$ to be zero on the zero section. Since the derivatives converge to zero at the zero section, by Lemma \ref{qmmg} $G^\alpha$ has strong differential zero at the zero section, and so $G^\alpha$ is locally Lipschitz on $TM$.

Now, let us consider the derivatives of $\frac{\p}{\p x^\beta}G^\alpha$, $\frac{\p}{\p y^\beta}G^\alpha$ on $TM\backslash 0$ extending them to the zero section as $0$. Their derivative, namely the second derivatives of $G^\alpha$,  $\frac{\p^2}{\p x^\beta\p x^\gamma}G^\alpha$, $\frac{\p^2}{\p y^\beta\p x^\gamma}G^\alpha$, $\frac{\p^2}{\p y^\beta\p y^\gamma}G^\alpha$, are positive homogeneous of degree 2, 1 and 0, respectively. As they are bounded on the $h$-unit bundle over a compact neighborhood of $p$, they are bounded over the corresponding ball (minus the zero section).

An application of the Lagrange mean value theorem, going in a way similar to what has been done in the proof of Lemma \ref{qmmg}, see also \cite[Thm.\ 4]{furi91}, implies that the derivatives of $G^\alpha$ are Lipschitz in a neighborhood of the zero section.
\end{proof}

\begin{lemma} \label{kktr}
Let us consider a
 spray  which is smooth on $TM\backslash 0$, and let us consider the exponential map bijection $\exp_p: O \to N$, $O\subset T_pM$, where $N$ is a normal open neighborhood of $p$. Then for some choice of $O$ and $N$, $\exp_p$ and $\exp_p^{-1}$ are $C^{1,1}$ and actually $C^\infty$ outside the origin (resp.\ $p$).\\
\end{lemma}

Restricted to the only direction $\exp_p$ and to Finsler sprays the result appeared before \cite[Prop.\ 4.1]{koehler14} where it is attributed to S. Ivanov. The weaker $C^1$ result goes back at least to Warner \cite[Thm.\ 4.5]{warner65}.

\begin{proof}
The exponential map is determined by the following system
\begin{align*}
\dot x^\alpha&=y^\alpha, \\
\dot y^\beta&=-2G^\beta(x,y).
\end{align*}
By ODE theory it is well known that if the spray is $C^k$, $k\ge 1$, then the exponential map is $C^k$ (this whether we consider the restriction outside the zero section or not). The exponential map is strongly differentiable at the zero section with the strong differential being the identity \cite{minguzzi13d} (hence it is invertible). This implies that the Jacobian of the exponential map will be invertible in a neighborhood of the zero section. This also implies that if the Jacobian, namely $(\exp_p)_*$ is Lipschitz so is  $(\exp_p^{-1})_*$ (use e.g.\ formula \cite[Prop.\ 2.5]{howard97}). The fact that a $C^{1,1}$ spray leads to a $C^{1,1}$ exponential map with $C^{1,1}$ inverse was  noted in  \cite[Remark 6]{minguzzi13d}.
\end{proof}

\begin{theorem} \label{gwpo}
Let $(M,L)$ be a Finsler spacetime and let $N$ be a normal open neighborhood of $p$. Let $f,g: N \to \mathbb{R}$,
\[
f(q):= F(p, \exp^{-1}_p q), \qquad l(q):= -2L(p, \exp^{-1}_p q).
\]
For every smooth curve $\gamma: [0,1]\to N$, $\gamma(0)=p$,  $\dot \gamma(0)=v\ne 0$, the functions $f(\gamma(t))$ and $l(\gamma(t))$ are respectively $C^0$ (actually $C^1$ if $\gamma$ is timelike) and $C^2$ on some interval $[0,\epsilon)$, $\epsilon>0$, and
\begin{align}\label{cmtq1}
F(p,v)&=\frac{\dd}{\dd t}  f(\gamma(t))\vert_{t=0} =\lim_{t\to 0^+} \frac{1}{t} f(\gamma(t)) , \\
-2L(p,v)&=\frac{1}{2}\frac{\dd^2}{\dd t^2}  l(\gamma(t))\vert_{t=0} = \frac{1}{2} \lim_{t\to 0^+}  \frac{1}{t} \frac{\dd}{\dd t}   l(\gamma(t))  . \label{cmtq}
\end{align}
\end{theorem}

\begin{proof}

Let $\gamma: [0,1]\to M$, $\gamma(0)=p$, $v:=\dot \gamma(0)\ne 0$  be any smooth curve , and let $v_t=\exp^{-1}_p \gamma(t)$. Note that $\lim_{t\to 0^+} \frac{1}{t} v_t=(\exp^{-1}_p)_* v$. But $(\exp_p)_*: TT_p M\to T_p M$ (we can identify $TT_pM$ with $T_pM$) is well known to be the identity \cite[Thm.\ 11.1.1]{shen01} \cite[5.7.3]{bao00} \cite[Eq.\ (39)]{minguzzi13d} (this result is independent of the signature of the Hessian metric, in fact it depends only on the spray), and so it is the inverse $(\exp^{-1}_p)_*$. We conclude $\lim_{t\to 0^+} \frac{1}{t} v_t=v$ and so
\[
\lim_{t\to 0^+} \frac{1}{t} f(\gamma(t))=\lim_{t\to 0^+} F(p, \frac{1}{t} v_t)=F(p, v),
\]
where we used the positive homogeneity and continuity of $F$. Note that if $\gamma$ is timelike then $v_t$ is $C^1$ and timelike and so $F(p,v_t)$ is $C^1$ because  $F$  is evaluated on the smooth region.

Let us prove the second equation. From \cite[Thm.\ 5]{minguzzi13d} (but also from Lemma \ref{kktr}) we know that $l: N \to \mathbb{R}$ is $C^{1,1}$ with differential at $q$, $\dd  l = - 2g_{P(p,q)}(P(p,q),\cdot)$ where $P(p,q)$, called {\em position vector}, is the tangent vector at $q$ of the geodesic $\sigma_{p,q}: [0,1] \to N$, $\sigma_{p,q}(0)=p$, $\sigma_{p,q}(1)=q$. Thus $P(p,\gamma(t))=(\exp_p)_* v_t$ and since $\exp_p$ is $C^1$, $(\exp_p)_*$ is $C^0$ and so   $P(p,\gamma(t))$ converges to zero and $\lim_{t\to 0^+} \frac{1}{t} P(p,\gamma(t))=\lim_{t\to 0^+}  (\exp_p)_* (\frac{1}{t}v_t)=(\exp_p)_* v=v$.

The function 
\[
h(t):=\frac{\dd}{\dd t} l(\gamma(t))=\langle \dd l, \dot \gamma(t)\rangle=-2g_{P(p,\gamma(t))}(P(p,\gamma(t)),\dot\gamma(t))
\]
is $C^0$ and vanishes for $t=0$ as $P(p,\gamma(0))=0$ (note that  $g_{P(p,\gamma(t))/t}(P(p,\gamma(t))/t,\cdot)$ $ \to g_v(v,\cdot)$ for $t\to 0$). A standard result of real analysis which follows from L'Hôpital's rule  states that if for a continuous function $h:[0,b)\to \mathbb{R}$, $h'(t)$ exists for $t>0$ and converges for $t\to 0^+$ then $h'(0)$ exists and coincides with that limit. For sufficiently small $t>0$ (we just want to restrict ourselves to a segment that does not pass twice from $p$),  we can write  $l(\gamma(t))=-2L(p, \exp_p^{-1}(\gamma(t)))$ where both $-2L(p, \cdot)$ and $\exp_p^{-1}$ are twice differentiable in the point of interest (as the former function is evaluated over non-zero vectors, while the regularity of the latter function is smooth outside the zero section).  Recalling that $v_t=\exp_p^{-1} \gamma(t)$, we have by the $C^{1,1}$ regularity of $\exp_p^{-1}$ near the zero section that $v_t$ is $C^{1,1}$ in a neighborhood of $t=0$, and so $\ddot v_t$ remains bounded as $t$ approaches zero. For sufficiently small $t>0$,
\begin{align*}
\frac{\dd h}{\dd t}&= -2\frac{\dd^2}{\dd t^2}L(p, v_t)=-2 \Big\{\frac{\partial L}{\partial y^{\mu}\partial y^{\nu}}\dot{v}^{\mu}_t\dot{v}^{\nu}_t+\frac{\p L}{\p y^\mu}(p,v_t) \ddot v_t^\mu\Big\}\\
&=-2 \Big\{g_{v_t}(\dot v_t,\dot v_t) + \frac{\p L}{\p y^\mu}(p,v_t) \ddot v_t^\mu\Big\}
\end{align*}
The second term vanishes for $t\to 0^+$ because $L$ is $C^1$ and $v_t\to 0$, while $\ddot v_t$ remains bounded. Now, due to $v_t$ being $C^1$,  $\lim_{t\to 0^+} \dot v_t=\dot v_0= \lim_{t\to 0^+} \frac{1}{t} v_t=v$, thus $g_{v_t}(\dot v_t,\dot v_t)=g_{v_t/t}(\dot v_t,\dot v_t)\to g_{v}(v,v)=2L(p,v)$, and we conclude that $\lim_{t\to 0^+}h'(t)=-4 L(p,v)$ thus
\[
\frac{\dd^2}{\dd t^2}  l(\gamma(t))\vert_{t=0}= -4L(p,\dot \gamma(0)).
\]
\end{proof}

\begin{theorem} \label{cboo}
Let $(M,L)$ be a future distinguishing Lorentz-Finsler spacetime and let $d$ be the associated Lorentz-Finsler distance, then for every smooth curve $\gamma: [0,1]\to M$, $\gamma(0)=p$,  $\dot \gamma(0)=v\ne 0$, with image in $J^+(p)$ (for instance, this is the case if $\gamma$ is a continuous causal curve and $v$ is causal \cite[Thm.\ 6]{minguzzi13d})
 the function $d_p(\gamma(t))$ and $d_p^2(\gamma(t))$ are respectively $C^0$ (actually $C^1$ if $\gamma$ is timelike) and $C^2$ on an interval $[0,\epsilon)$, $\epsilon>0$, and
\begin{align}\label{cmtq1b}
F(p,v)&=\frac{\dd}{\dd t}  d_p(\gamma(t))\vert_{t=0} =\lim_{t\to 0^+} \frac{1}{t} d_p(\gamma(t)) , \\
F^2(p,v)&=\frac{\dd^2}{\dd t^2} \frac{1}{2} d^2_p(\gamma(t))\vert_{t=0} =\frac{1}{2} \lim_{t\to 0^+}  \frac{1}{t} \frac{\dd}{\dd t}  d_p^2(\gamma(t)) . \label{cmtqb}
\end{align}
\end{theorem}

The former formula is the Lorentz-Finsler generalization of the Busemann-Mayer formula, while the latter formula seems to be new also in positive signature, see Thm.\ \ref{cnww}.

The theorem states that  the Finsler fundamental function can be recovered from the Lorentz-Finsler distance $d$. Of course,  analogous past statements hold.


For the validity of the second equation $L$ must  be $C^2$ on the slit tangent bundle and so have finite Lorentzian Hessian also at the boundary of the future timelike cone, but need not be defined on the whole slit tangent bundle. In the proof of the latter equation  the second derivative of $L$ is only invoked outside the zero section.

\begin{proof}
Let $U$ be a distinguishing neighborhood for $p$ contained in a normal neighborhood $N$ (the definition of distinguishing neighborhood is the same as for Lorentzian geometry, see \cite[Thm.\ 4.44(iii)]{minguzzi18b}, for the existence of normal neighborhood, see \cite[Thm.\ 4]{minguzzi13d}). Every future directed causal curve starting from $p$ an ending in $U$ must necessarily be contained in $U$ and so in the normal neighborhood $N$. This means that for $q\in J^+(p)\cap U$, $d(p,q)=d_U(p,q)=\tilde d_p(q)$ where $\tilde d$ is the Lorentz-Finsler distance of the normal neighborhood.
Since for sufficiently small $t$, $q:=\gamma(t)\in J^+(p,N)$, by \cite[Thm.\ 6]{minguzzi13d}  $\exp_p^{-1} q$ is causal and so  $\tilde d_p(q)=F(p,\exp_p^{-1} q)$, thus for  $q\in J^+(p)\cap U$, $d_p(q)=F(p,\exp_p^{-1} q)$, $d^2_p(q)=-2L(p,\exp_p^{-1} q)$.
 The result now follows from Thm.\ \ref{gwpo}.
\end{proof}

\begin{theorem}
Let $(M,L)$  and $(M',L')$ be  strongly causal smooth Finsler spacetimes and let $f: M \to M'$ be a distance preserving bijection:
\[
d'(f(p),f(q))=d(p,q).
\]
 Then $M$ and $M'$  have the same dimension, $f$ is a (smooth) diffeomorphism such that $f_*$ sends at every point the future timelike cone into the future timelike cone, and for every $p\in M$, and causal vector $v\in T_pM$
 \[
 F(p,v)=F'(f(p),f_*(v)).
 \]
\end{theorem}


\begin{proof}
The proof  requires minimal modifications with respect to the isotropic case.
We need to change just the third paragraph that reads as follows:

\begin{quote}
This implies that $d_{q_i}(r)=F(q_i, \exp_{q_i}^{-1}(r))$, which  is smooth in a neighborhood of $p$ as $\exp_{q_i}^{-1}(r)$ is timelike. Its differential at $p$ is $g_{u_i}(u_i, \cdot)$ where $u_i:=\hat t_i$ is a normalized past-directed timelike vector, see e.g.\ \cite[Thm.\ 5]{minguzzi13d}.

The Legendre map $\ell: T_pM\backslash 0 \to T_p^*M\backslash 0$, $v\mapsto g_v(v,\cdot)$ has a Jacobian whose determinant is that of $g_v$, thus it is invertible and a diffeomorphism near some chosen normalized past-directed timelike vector $v$, which implies that if vectors $t_i$ are chosen so that $u_i$ are close to $v$, the images $g_{u_i}(u_i, \cdot)$ will be linearly independent.
The map $V\to \mathbb{R}^{n+1}$, $r \mapsto (d_{q_1}(r), \ldots, d_{q_{n+1}}(r))$ has thus non-singular Jacobian at $p$.
\end{quote}

The sixth paragraph changes as follows:

\begin{quote}
Thus if $v$ is future-directed timelike $v':=f_* v$ is future-directed timelike.  Let $\gamma : I\to M$ be a smooth timelike curve starting from $p\in M$, $p=\gamma(0)$, with timelike tangent $v\in T_pM$,  then by Eq.\  (\ref{cmtqb}), setting $\gamma'=f \circ \gamma$, we observe that by the equality $d_p=d_p' \circ f$,  $\gamma'$ has image contained in $I^+(p')$ thus Theorem \ref{cboo} can be applied to it as well as to $\gamma$
\[
-2L(p,v)=\frac{\dd^2}{\dd t^2} \frac{1}{2} d^2_p(\gamma(t))\vert_{t=0}=\frac{\dd^2}{\dd t^2} \frac{1}{2} d'{}^2_{p'}(\gamma'(t))\vert_{t=0}=-2L'(p',v'),
\]
which shows that $v'=f_*(v)$ is timelike.
By continuity, the just found identity holds also for $v$ causal.
\end{quote}
\end{proof}

\subsubsection{The positive signature Finsler case}
We give also the statements analogous to the above for a Finsler space but we omit the proof since  simpler. Here $F$ is defined by $L=\frac{1}{2} F^2$ over the whole $TM$.\\

\begin{theorem} \label{cnww}
Let $(M,L)$ be a Finsler space and let $N$ be a normal open neighborhood of $p$. Let $d_p: N \to \mathbb{R}$,
\[
d_p(q):= F(p, \exp^{-1}_p q).
\]
For every smooth curve $\gamma: [0,1]\to N$, $\gamma(0)=p$,  $\dot \gamma(0)=v\ne 0$, the functions $d_p(\gamma(t))$ and $d_p^2(\gamma(t))$ are respectively $C^1$  and $C^2$ on some interval $[0,\epsilon)$, $\epsilon>0$, and
\begin{align}\label{cmtq1a}
F(p,v)&=\frac{\dd}{\dd t}  d_p(\gamma(t))\vert_{t=0} =\lim_{t\to 0^+} \frac{1}{t} d_p(\gamma(t)) , \\
2L(p,v)&=\frac{1}{2}\frac{\dd^2}{\dd t^2}  d_p^2(\gamma(t))\vert_{t=0} = \frac{1}{2} \lim_{t\to 0^+}  \frac{1}{t} \frac{\dd}{\dd t}   d_p^2(\gamma(t))  . \label{cmtqa}
\end{align}
\end{theorem}

We arrive at the following interesting result. It seems  to us that the proof in \cite[Prop.\ 11.3.3]{shen01} tacitly uses the additional assumption that $\exp$,  there denoted $\varphi^{-1}$, is $C^2$. It is known that if the exponential map is $C^2$   the Finsler space is Berwald \cite{akbarzadeh88}.\\

\begin{corollary}
Let $(M,L)$ be a Finsler space  such that $d_p^2$ is $C^2$ in a neighborhood of $p$ for every $p$, then $(M,L)$ is necessarily Riemannian.
\end{corollary}

\begin{proof}
Equation (\ref{cmtq1a}) multiplied by $d_p(\gamma(t))$ proves, by the arbitrariness of $\dot \gamma(0)$, that $d_p^2$ has vanishing differential at $p$. Thus Eq.\ (\ref{cmtqa})  can be rewritten with the $C^2$ assumption on  $d_p^2$, $4L(p,v)= \frac{\p^2 d_p^2}{\p x^\mu \p x^\nu}(p) v^\mu v^\nu$ which is quadratic.
\end{proof}

We include the following  known consequence of the Busemann-Mayer formula for completeness.\\

\begin{theorem}
Let $(M,L)$  and $(M',L')$ be  smooth Finsler spaces and let $f: M \to M'$ be a distance preserving bijection:
\[
d'(f(p),f(q))=d(p,q).
\]
 Then $M$ and $M'$  have the same dimension, $f$ is a (smooth) diffeomorphism  and for every $p\in M$, and vector $v\in T_pM$
 \[
 L(p,v)=L'(f(p),f_*(v)).
 \]
\end{theorem}
Again, the proof is analogous to those given above for the Lorentz-Finsler case but simpler.

\section{Low regularity} \label{qqpo}
In this section we discuss to what extent the results of Section \ref{cnpda}, developed for smooth spacetimes $(M,g)$, generalize to less regular structures.

\subsection{Equivalence of distinction/reflectivity concepts with their $d$-version}
Proposition \ref{cnqxz} and its consequences, Corollaries \ref{moprs} and  \ref{mopq}, are formulated for Lorentzian manifolds but their proof applies in much greater generality.

\subsubsection{Proper Lorentz-Finsler spaces}
For proper Lorentz-Finsler spaces \cite{minguzzi17} we have the generalization\\

\begin{proposition} \label{clqx}
 Let $(M,C,\mathscr{F})$ be a locally Lipschitz proper Lorentz-Finsler space such that $\mathscr{F}(\p C)=0$. For any two events $p,q\in M$, the inclusion $I^+(q)\subset I^+(p)$ is equivalent to the property $d_q\le d_p$.
Dually, the   inclusion $I^-(p)\subset I^-(q)$ is equivalent to the property $d^p\le d^q$.\\
\end{proposition}

\begin{proof}
The proof passes unaltered to this case due to the identity $I=\{d>0\}$ (proved in the proof of Thm.\ \ref{ckkq})   and from \cite[Thm.\ 2.56]{minguzzi17}  which is applied  in the proof when establishing the existence of the timelike curve $\gamma$.
\end{proof}

In the context of locally Lipschitz proper cone structures future distinction can be understood as the following property, see the proof of \cite[Thm.\ 2.18]{minguzzi17}, (antisymmetry of $D_f$) `$y\in \overline{I^+(x)}$ and $x\in \overline{I^+(y)}$ implies $x=y$', which is easily shown to be equivalent to the standard property $I^+(x)=I^+(y) \Rightarrow x=y$.

This implies, via the same proofs as in the smooth case,\\

\begin{corollary}
 Let $(M,C,\mathscr{F})$ be a locally Lipschitz proper Lorentz-Finsler space such that $\mathscr{F}(\p C)=0$. Then (fu\-ture/past/weak) $d$-distinction is equivalent to  (resp.\ fu\-ture/past /weak)  distinction. Similarly,  (future/past) $d$-reflectivity is equivalent to  (resp.\ future/past)  reflectivity.\\
\end{corollary}


\begin{remark}
Similarly, an inspection shows that, for the structure  $(M,C,\mathscr{F})$ of a locally Lipschitz proper Lorentz-Finsler space such that $\mathscr{F}(\p C)=0$, we have, with the same proofs, the analogs of  Prop.\ \ref{nnrt},  Prop.\ \ref{bpmf}, Thm.\ \ref{clpqo}, Cor.\  \ref{coer}, Thm.\ \ref{cjpr}, Cor.\ \ref{cfpr}, Thm.\ \ref{mxdf} and Thm.\ \ref{coox}.\\
\end{remark}

\subsubsection{Lorentzian length spaces}

Lorentzian metric spaces can be endowed with a natural causal relation though which  isocausal curves can be defined \cite{minguzzi22,minguzzi24b}. By assuming that every two chronologically related points are connected by a (maximal) isocausal curve the {\em Lorentzian (resp.\ length) prelength spaces} are obtained \cite[Def.\ 5.13]{minguzzi22} \cite[Def.\ 5.1]{minguzzi24b}.\\

\begin{proposition} \label{ckqp}
 Let $(X,d)$ be a Lorentzian length space. For any two events $p,q\in X$, the inclusion $I^+(q)\subset I^+(p)$ is equivalent to the property $d_q\le d_p$.
Dually, the   inclusion $I^-(p)\subset I^-(q)$ is equivalent to the property $d^p\le d^q$.\\
\end{proposition}

Note that the point $\gamma(s)$ is close to $q$ in the smooth version of the proof. Here we do not have this result but the proof works anyway. Under global hyperbolicity, due to Thm.\ \ref{ckkq}, Proposition \ref{clqx} would be a consequence of this one.

\begin{proof}
Suppose that  $I^+(q)\subset I^+(p)$  and let $r \in X$. If $r\notin I^+(q)$ then $d_q(r)=0$ and the inequality we wish to prove follows. If $r\in I^+(q)$ there is a maximal isocausal curve $\gamma:[0,1]\to X$ connecting $q=\gamma(0)$ to $r=\gamma(1)$ (note that it can have a pair of points at zero distance).   Since both $\gamma$ and $d$ are continuous in the topology of the Lorentzian metric space $X$, for any $0<m<d(q,r)$ we have $d(\gamma(s), r)=m$ for some  $s$. By maximality of the isocausal curve we have $d(q,\gamma(s))=d(q,r)-d(\gamma(s), r)=d(q,r)-m>0$. Thus  $\gamma(s)\gg q$ hence $\gamma(s)\gg p$, which implies
\[
d(p,r)\ge d(p,\gamma(s))+d(\gamma(s),r) \ge d(\gamma(s),r) =m.
\]
 Since $m<d_q(r)$ is arbitrary, $d_p(r)\ge d_q(r)$. By the arbitrariness of $r$,  $d_q\le d_p$.

For the converse, suppose that  $d_q\le d_p$, then if $r\in I^+(q)$ we have $d_q(r)>0$ so $d_p(r)>0$ which implies $r\in I^+(p)$. By the arbitrariness of $r$, $I^+(q)\subset I^+(p)$.
\end{proof}
This implies, via the same proofs as in the smooth case,\\

\begin{corollary} \label{cjnrx}
 Let $(X,d)$ be a Lorentzian length space. Then (fu\-ture/past/weak) $d$-distinction is equivalent to  (resp.\ fu\-ture/past /weak)  distinction. Similarly,  (future/past) $d$-reflectivity is equivalent to  (resp.\ future/past)  reflectivity.\\
\end{corollary}

On every Lorentzian metric space we have a length functional $L:\sigma\to L(\sigma)$, cf.\  \cite[Def.\ 6.1]{minguzzi22} \cite[Def.\ 5.12]{minguzzi24b}.
A different proof of Prop.\ \ref{ckqp}, more similar to that given for the smooth case, could have been obtained using the following result.\\

\begin{proposition}
Let $(X,d)$ be a Lorentzian metric space, and let $\sigma:[0,1] \to X$ be an isocausal curve connecting $x$ and $y$, $x\ll y$. Then the function $L(\sigma(s))$ is continuous.\\
\end{proposition}

\begin{proof}
By definition of Lorentzian length, $0\le L(\sigma(s'))-L(\sigma(s))\le d(\sigma(s),\sigma(s'))$ for $s'\ge s$, thus $\vert L(\sigma(s'))-L(\sigma(s))\vert\le d(\sigma(s),\sigma(s'))+d(\sigma(s'),\sigma(s))$. As both $d$ and $\sigma$ are continuous, fixing $s$ and taking $s'\to s$ shows that $L(\sigma(s))$ is continuous.
\end{proof}

\subsection{Lorentzian distance in abstract setting}
In this section, unless otherwise stated, $(X,d)$ is just a set endowed with a function  $d: X\times X \to [0,+\infty]$ which satisfies the reverse triangle inequality on chronologically related triples (the chronological relation being $I:=\{d>0\}$).\\

\begin{proposition}\label{upcont-refl}
    Let $(X,d)$ be a pair as above endowed with a topology $T$. Suppose that in such topology
      $p\in \overline{I^{\pm}(p)}$ for every $p\in X$  and the functions $d_r$, $r\in X$, are  upper semi-continuous, then  the following property holds
      \begin{equation}
      I^+(p)\supset I^+(q) \Rightarrow d^p \le d^q,
      \end{equation}
      which implies both  past reflectivity and past $d$-reflectivity.
      Similarly, the upper semi-continuity of the functions $d^r$, $r\in X$,  implies  $I^-(p)\subset I^{-}(q)\Rightarrow d_q\le d_p$, and hence both future reflectivity and future $d$-reflectivity. \\
\end{proposition}

The statement on the fact that the found implication is stronger than past reflectivity and past $d$-reflectivity follows from the trivial implications $d_q\le d_p\Rightarrow I^+(p)\supset I^+(q)$ and $d^p \le d^q\Rightarrow I^-(p)\subset I^{-}(q)$.

We call the property $I^+(p)\supset I^+(q) \Rightarrow d^p \le d^q$ {\em strong past reflectivity}, and the property $I^-(p)\subset I^{-}(q)\Rightarrow d_q\le d_p$ {\em strong future reflectivity}. Due to Prop.\ \ref{ckqp} the reflectivity/$d$-reflectivity/strong reflectivity properties, declined also in the past/future versions, are  all equivalent for Lorentzian length spaces, see also  Cor.\ \ref{cjnrx}.

\begin{proof}
 Assume $I^+(p)\supset I^+(q)$ and let $r\in X$. If $d^p(r)=0$ there is nothing to prove, thus suppose $r\ll p$.
 By the upper $T$-semi-continuity of the function $d_r$ for any $m>d(r,q)$ there is a $T$-open set $O\ni q$, such that $d(r,q')<m$ for every $q'\in O$. Let us choose $q'\in I^+(q)\cap O$. Thus $d_q(q')>0$,  which implies by the assumed inclusion  $d_p(q')>0$. By the reverse triangle inequality
 \[
 m>d(r,q')\ge d(r,p)+d(p,q')\ge d(r,p)=d^p(r).
 \]
 By the arbitrariness of $m$, $d^p(r)\le d^q(r)$, which concludes the proof.
\end{proof}

So, reflectivity/$d$-reflectivity follows from upper semi-continuity in any topology in which every point has points in the chronological past and future  in any neighborhood, it does not have to be the manifold topology. In particular, all Lorentzian metric spaces with this property are ($d$-)reflective. As a special case, all Lorentzian prelength spaces, with any pair of chronologically ordered points being connected by an isochronal line, are $d$-reflective.

The proof of the next result coincides with that for the smooth case, cf.\ Thm.\ \ref{clpqo}.\\

\begin{theorem}\label{upcoce}
    Let $(X,d)$ be a set endowed with a function  $d: X\times X \to [0,+\infty]$ which satisfies the reverse triangle inequality on chronologically related triples. Let it  be  endowed with a topology $T$. Suppose that in such topology
      $p\in \overline{I^{\pm}(p)}$ for every $p\in X$ and $d$ is lower $T$-semi-continuous.  The function $d$ is $T\times T$-continuous if and only if the functions $d_p$ and $d^p$ are $T$-continuous for every $p\in X$.\\
\end{theorem}

The assumption that $d$ is lower $T$-semi-continuous in the first sentence can be dropped for a Lorentzian length space. Indeed, the last statement implies that $T$ contains the Alexandrov topology, and from there and the existence of a maximal curve (not necessarily isochronal) the  lower $T$-semi-continuity follows.

The idea behind the following proof is it show how reflectivity works  without using the manifold assumption.\\

\begin{proposition}\label{refl-cont}
Let $(X,d)$ be a set endowed with a function  $d: X\times X \to [0,+\infty]$ which satisfies the reverse triangle inequality on chronologically related triples.    Assume that weak $d$-distinction hold for $(X,d)$.  Suppose that for every $r\in X$ there are $x,y\in X$ such that $r\in I(x,y)$. Let $T$ be a topology of $X$, making $d_p$, $d^p$ $T$-continuous for each $p\in X$. Assume that for every $p,q\in X$ the diamond $I(p,q)$ is $T$-relatively compact.\footnote{These assumptions would be the same of a Lorentzian metric space if we were assuming that  $d$  is $T$-continuous and finite.}
   Then $T$ is the initial topology of the family $\{d_p, d^p, p \in X\}$.
   Moreover, if $d$-reflectivity holds any topology $T'$ such that $d_p$, $d^p$, for $p\in X$, are lower $T'$-semi-continuous contains $T$ (so $d_p$, $d^p$ are actually $T'$-continuous).
\end{proposition}

\begin{proof}
The first claim is proved in \cite{minguzzi22} using Prop.\ \ref{jept} (note that the initial topology of the family $\{d_p, d^p, p \in X\}$  is Hausdorff by weak $d$-distinction  and it contains the Alexandrov topology. Its open set are the elements of $\mathcal{A}$ in that result).

For the second statement, we are preliminarly going to show that  $d_p$, $d^p$ are $T'$-continuous for every $p$.
Assume that $d_p$ is not upper $T'$-semi-continuous at some point $q$.
Take $x,y\in X$ so that $x\ll q \ll y$. Then,
by the lower $T'$-semi-continuity of $d_x$ and $d^y$, $I^+(x)$ and $I^-(y)$ are $T'$-open and hence    $I(x,y)$ is $T'$-open.

Let $\Lambda$ be the family of $T'$-neighborhoods of $q$ contained in  $I(x,y)$ (it is non-empty as $I(x,y)$ belongs to it, further any neighborhood can be intersected with $I(x,y)$ to give another neighborhood, thus $\Lambda$ generates the filter of neighborhoods for $q$).

There is some $\epsilon>0$ such that for every  $O\in \Lambda$ there exists some $q'(O)\in O\subset I(x,y)$ such that  $d_p(q'(O))-d_p(q)>\epsilon$.
The family $\Lambda$ is a directed set by the reverse inclusion,
thus $q' :\Lambda \to I(x,y)$ is a net (we shall also write $q'_\lambda$ instead of $q'(O)$). Note that by construction, $q'_\lambda \to_{T'} q$.
By the upper $T$-semi-continuity of $d$ we can find an $T$-open set $U\ni q$ such that on it $d_p(r)-d_p(q)<\epsilon$ for every $r\in U$, thus for every $O\in \Lambda$, $q'(O)\notin U$. We conclude that  $q' :\Lambda \to \textrm{Cl}_{T}(I(x,y))\backslash U$ is a net with image in a $T$-compact set, thus there is a cluster point $\tilde q$ and hence a subnet converging to $\tilde q \in \textrm{Cl}_{T}(I(x,y))\backslash U$, $q'_\lambda \to_{T} \tilde q$. In particular, as $\tilde q\notin U$, $\tilde q\ne q$.




    Now, consider any $z\in X$. By lower $T'$-semi-continuity of $d_z$ and $T$-continuity of $d_z$  we have
    \[
    d(z,q)\leq \lim_{\lambda}d(z,q'_\lambda)=d(z,\tilde q),
    \]
    so $d^q\leq d^{\tilde q}$, and thus, by $d$-reflectivity $d_q \geq d_{\tilde q}$. Moreover, again by lower $T'$-semi-continuity of $d^u$ and $T$-continuity of $d^u$, for every $u\in X$,
    \[
    d(q,u)\leq \lim_{\lambda} d(q'_\lambda,u)=d(\tilde q,u),
    \]
    so $d_q\leq d_{\tilde q}$, and thus
$d^q\geq d^{\tilde q}$. We conclude that $d_q=d_{\tilde q}$ and $d^q=d^{\tilde q}$, therefore $q=\tilde q$ by weak distinction, a contradiction which shows that $d_p$ is $T'$-continuous. Similarly, $d^p$ is $T'$-continuous. This means that $T'$ is finer than the initial topology of these functions which coincides with $T$.
\end{proof}

So, reflectivity is equivalent to upper semi-continuity (with lower semi-con\-ti\-nuity assumed) in a very general setting.

We recall that we consider only Lorentzian metric spaces without chronological boundary, so that every point is contained in some chronological diamond.\\

\begin{corollary}\label{refl-cont-lms}
    Let $(X,d)$ be a $d$-reflective  Lorentzian metric space, and let $T$ be any topology with respect to which  $d_p$, $d^p$ are lower $T$-semi-continuous. Then $d_p$, $d^p$ are $T$-continuous for each $p\in X$.\\
\end{corollary}

The property of $d$-reflectivity implies the coincidence of the Alexandrov and Lorentzian metric space topologies.\\

\begin{corollary}
    Let $(X,d)$ be a  $d$-reflective Lorentzian metric space. Assume that $d_p, d^p$ for $p\in X$ are lower semi-continuous in the Alexandrov topology (e.g.\ $X$ is a manifold and $d$ is induced by a Lorentzian metric on $X$). Then the Alexandrov topology coincides with the initial topology of the functions $\{d_p, d^p, p\in X\}$, that is, with the Lorentzian metric space topology (which is Hausdorff).
\end{corollary}

\begin{proof}
    By Corollary \ref{refl-cont-lms}, $d_p$, $d^p$ are continuous in the Alexandrov topology for every $p\in X$, thus the Alexandrov topology is finer that the initial topology of the functions $\{d_p, d^p, p\in X\}$. It is clear that the initial topology of the functions $\{d_p, d^p, p\in X\}$ is finer than the Alexandrov topology.
\end{proof}

\begin{theorem} \label{tnbod}
    Let $(X,d)$ be a Lorentzian metric space and let $T$ be a topology for $X$ such that  $p\in \overline{I^{\pm}(p)}$ for every $p\in X$, and  such that the function $d$ is lower $T$-semi-continuous. Then the following statements are equivalent:
    \begin{itemize}
        \item[(i)] $d$ is $T\times T$-continuous;
        \item[(ii)] The functions $d_p$ and $d^p$ are $T$-continuous for every $p\in X$;
        \item[(iii)] $T$ contains the Lorentzian metric space topology of $X$;
        \item[(iv)] $(X,d)$ is $d$-reflective;
    \end{itemize}
    in which case strong reflectivity holds.
\end{theorem}

\begin{proof}
    The first two statements are equivalent  by Theorem \ref{upcoce}. The second is the same as the third one because the topology of the  Lorentzian metric space $X$ is the initial topology of the functions $d_p$, $d^p$. The second item implies the last one by Proposition \ref{upcont-refl}, and the converse implication is due to Proposition \ref{refl-cont}. The last statement follows from Prop.\ \ref{upcont-refl}.
\end{proof}

\begin{theorem}
In a Lorentzian length space $(X,d)$ function $d$ is lower semi-continuous in the Alexandrov topology. Moreover, the conditions
\begin{itemize}
\item[(a)] $p\in \overline{I^{\pm}(p)}$ for every $p\in X$ (closure is in the Lorentzian metric space topology),
\item[(b)] The chronological diamonds provide a basis for the Lorentzian metric space topology,
\end{itemize}
are equivalent. They imply strong reflectivity (equiv.\ ($d$-)reflectivity) which implies the equivalence of the Lorentzian metric space topology with the Alexandrov topology.
\end{theorem}


\begin{proof}
Let $p\ll q$ otherwise lower semi-continuity of $d$ at $(p,q)$ is clear. Let $0<m<d(p,q)$ and let $\sigma:[0,1]\to X$ be a maximal isocausal curve connecting $p$ to $q$. By the continuity of $d_p\circ \sigma$ and $d^q\circ \sigma$ we can find $p',q'\in \sigma$, such that $d_p(p')=(d(p,q)-m)/2=d^q(q')$, so that $d(p',q')=m$. Then $(p,q)\in I^{-}(p')\times I^+(q')$ which is open in the Alexandrov topology and for every $(p'',q'')\in   I^{-}(p')\times I^+(q')$, we have by the reverse triangle inequality $d(p'',q'')\ge d(p',q')=m$.

(a) $\Rightarrow$ (b). This follows from  \cite[Thm.\ 3.9]{minguzzi24b}.

(b) $\Rightarrow$ (a).  Let us prove $p\in \overline{I^{+}(p)}$, the other inclusion being analogous.
Let $O\ni p$ be a Lorentzian metric space open set.  By the assumption there are $w,z\in X$,  such that $p\in I(w,z)\subset O$.  Let   $\sigma$ be the  maximal isocausal curve connecting $p$ to $z$, then we can find $z'$ on it such that $d(p,z')=d(z',z)=d(p,z)/2>0$. Thus $z'\in I(w,z)\subset O$ and  $z'\in I^+(p)$ which proves that every open set $O$ of $p$ intersects $I^+(p)$.

Strong reflectivity follows from Prop.\ \ref{upcont-refl} with $T$ the Lorentzian metric space topology. As mentioned previously, due to Prop.\ \ref{ckqp} strong reflectivity is equivalent to ($d$-)reflectivity for Lorentzian length spaces.

Finally, by Cor.\ \ref{refl-cont-lms} $d$-reflectivity implies the equivalence of Lorentzian metric space  and Alexandrov topologies.
\end{proof}

%

\section{Conclusions}

The problem of defining a natural topology for spacetime and the challenge of establishing an abstract notion of a {\em Lorentzian metric space} are interconnected. This is because the latter must also be equipped with a topology derived solely from the two-point function $d$. In a smooth setting, this topology could, in principle, differ from that of the manifold. If the two were to coincide, it would be a significant result, as it would provide a way to define the ``manifold topology" from an abstract framework, without relying on the structure of a manifold. This is precisely what we have achieved in this work.

We have identified certain properties, specifically (i)-(iii) in Theorem \ref{cakop}, which are expressed exclusively in terms of the function \( d \). These properties ensure that the manifold topology (in the context of a Lorentzian manifold) aligns with the initial topology induced by the functions $d_p, d^p \in M$. We might thus refer to the canonical topology of the Lorentzian metric space as the ``manifold topology'' even if the structure $(M,d)$ is not that of a manifold.

 This outcome validates the approach to Lorentzian metric spaces based solely on the function \( d \), as introduced in \cite{minguzzi22,minguzzi24b}. The only caveat is that the notion of {\em weak $d$-distinction} in \cite{minguzzi22,minguzzi24b} must be replaced here with the slightly stronger condition of {\em future or past $d$-distinction}. We hope that future work will improve this result by demonstrating complete equivalence. Alternatively, the theory presented in \cite{minguzzi22,minguzzi24b} could be reformulated by strengthening the conditions for a Lorentzian metric space to match the conditions (i)-(iii) established in this work. A further possibility is that of adding a condition of $d$-reflectivity to the definition.

We achieved this result through a detailed study of the continuity properties of the Lorentzian distance, uncovering findings that appear to be novel even within the smooth framework. Along the way, we found it necessary to introduce and explore new properties, such as {\em $d$-distinction} and {\em $d$-reflectivity}, which, as we demonstrated, coincide with traditional causality concepts in a smooth setting.

We believe this work strengthens the case for a ``metric" type theory of spacetime based solely on the function $d$. Such a theory is not only feasible but may also prove useful in the development of a quantum spacetime theory, particularly in the pursuit of unifying gravity with the other fundamental forces.

\bmhead{Acknowledgements}
Useful conversations with Stefan Suhr are acknowledged. This study was funded by the European Union - NextGenerationEU, in the framework of the PRIN Project (title) {\em Contemporary perspectives on geometry and gravity} (code 2022JJ8KER - CUP B53D23009340006). The views and opinions expressed in this article are solely those of the authors and do not necessarily reflect those of the European Union, nor can the European Union be held responsible for them.\\



\noindent {\bf Confict of interest.} All authors declare that they have no conflict of interest.\\

\noindent
{\bf Data availability.} This manuscript has no associated data.


\end{document}